%% file: 0head.tex
\documentclass[12pt,a4paper]{amsart}
\usepackage{amsmath,amsfonts,amssymb,nameref}
\usepackage{tikz} 
\usepackage{color}
\usepackage[top=4cm, bottom=3cm, left=3cm, right=3cm]{geometry} 
\usepackage{graphicx}
\usetikzlibrary{calc,automata,matrix} 
\usetikzlibrary{decorations.markings}

\definecolor{mygrey}{RGB}{200,200,200}
\usepackage[pdftex]{hyperref}
\usepackage{soul}
\hypersetup{colorlinks,
citecolor=black,
filecolor=black,
linkcolor=black,
urlcolor=black}
\usepackage[all]{hypcap}
\newcommand{\N}	{\mathbb N}
\newcommand{\Z}	{\mathbb Z}
\newcommand{\R}	{\mathbb R}

\newcommand{\Cay}	{\operatorname{Cay}}

\newcommand{\supp}	{\operatorname{supp}}

\newcommand{\freeproduct}{*}

\newtheorem{thm}{Theorem}[section] 
\newtheorem{prop}[thm]{Proposition}
\newtheorem{lem}[thm]{Lemma} 
\newtheorem{cor}[thm]{Corollary}
\newtheorem*{thm*}{Theorem} 
\newtheorem*{prop*}{Proposition}
\newtheorem*{lem*}{Lemma} 
\newtheorem*{cor*}{Corollary}
\theoremstyle{definition}
\newtheorem{defi}[thm]{Definition}
\newtheorem{example}[thm]{Example}
\newtheorem{remark}[thm]{Remark}
\newtheorem*{defi*}{Definition}
\newtheorem*{example*}{Example}
\newtheorem*{remark*}{Remark}
\newtheorem*{problem*}{Problem}

\theoremstyle{plain}

\begin{document}
\tikzset{->-/.style={decoration={
  markings,
  mark=at position .6 with {\arrow{>}}},postaction={decorate}}}
\tikzset{-<-/.style={decoration={
  markings,
  mark=at position .6 with {\arrow{<}}},postaction={decorate}}}
\tikzset{->>-/.style={decoration={
  markings,
  mark=at position .7 with {\arrow{>}},
  mark=at position .5 with {\arrow{>}}},
  postaction={decorate}}}
\tikzset{-<<-/.style={decoration={
  markings,
  mark=at position .7 with {\arrow{<}},
  mark=at position .5 with {\arrow{<}}},
  postaction={decorate}}}

\title[$C(6)$-groups are SQ-universal]{Infinitely presented $C(6)$-groups are SQ-universal}
\author{Dominik Gruber}
 \thanks{The work is supported by the ERC grant of Prof. Goulnara Arzhantseva ``ANALYTIC" no. 259527.}
 \address{Department of Mathematics, University of Vienna, Oskar-Morgenstern-Platz 1, 1090 Vienna, Austria}
 \email{dominik.gruber@univie.ac.at}
\maketitle
\begin{abstract}
We prove that infinitely presented classical $C(6)$ small cancellation groups are SQ-universal. We extend the result to graphical $Gr_*(6)$-groups over free products. For every $p\in\N$, we construct uncountably many pairwise non-quasi-isometric groups that admit classical $C(p)$-presentations but no graphical $Gr'(\frac{1}{6})$-presentations.\end{abstract}
\input{1introduction}
\subsection{Acknowledgements}
 The author thanks Goulnara Arzhantseva, Christopher Cashen, and Markus Steenbock for helpful comments and discussions.
\input{3classical4}
\input{graphical_product2}
\input{undistorted_subgroups3}
\input{bibliography_arxiv}
\end{document}

%% file: 1introduction.tex
\section{Introduction}

Small cancellation theory is a rich source of examples and counterexamples of finitely generated infinite groups. Classical small cancellation theory, as defined in \cite[Chapter V]{LS}, has provided many hyperbolic groups and limits of hyperbolic groups. It has been generalized in various ways, and these generalizations have been used to study and construct famous groups such as Burnside groups \cite{Aid}, Tarski monsters \cite{Ols82}, and, more recently, Gromov's monsters \cite{Gro,AD,Osa}.

The class of $C(6)$ small cancellation groups is, in a sense, the largest nontrivial class of classical small cancellation groups. \emph{Every} group admits a $C(5)$-presentation.
While $C(6)$-groups are not necessarily limits of hyperbolic groups, they exhibit features of nonpositive curvature. For example, any $C(6)$-presentation where no relator is a proper power is \emph{aspherical}, i.e. the associated presentation complex is aspherical, see \cite{CCH} and \cite[Theorem 13.3]{Ols89}.

The major result of this paper is that all infinitely presented classical $C(6)$-groups are SQ-universal. A group $G$ is \emph{SQ-universal} if every countable group $C$ can be embedded in a quotient $Q$ of $G$. A countable SQ-universal group must have uncountably many pairwise non-isomorphic proper quotients. 
Moreover, any SQ-universal group contains non-abelian free subgroups and, in particular, is non-amenable. 
Our result applies to a host of concrete examples of infinitely presented small cancellation groups such as the infinite groups with no nontrivial finite quotients due to Pride \cite{Pride}, the uncountably many pairwise non-quasi-isometric 2-generated groups due to Bowditch \cite{Bow}, and the first groups with two non-homeomorphic asymptotic cones due to Thomas and Velickovic \cite{TV}. 

We extend our main result to infinitely presented graphical small cancellation groups over free products of groups. This is achieved by presenting a new viewpoint on graphical small cancellation presentations over free products that lets us naturally generalize our proofs with little technical effort. We moreover present the first construction of groups that distinguish the \emph{metric} $C'(\lambda)$-small cancellation condition from the \emph{non-metric} $C(p)$-condition by producing for every $p$ uncountably many $C(p)$-groups that do not admit any $C'(\frac{1}{6})$-presentation.

Many results about the SQ-universality of \emph{finitely presented} small cancellation groups are known: All non-elementary hyperbolic groups are SQ-universal \cite{Ols,Del}. Since any finitely presented $C(p)$-$T(q)$-group, where $\frac{1}{p}+\frac{1}{q}<\frac{1}{2}$, is hyperbolic, this result covers a large class of finitely presented small cancellation groups. Furthermore, the SQ-universality of finitely presented $C(3)$-$T(6)$-groups has been investigated with partial positive results \cite{How}. Al-Janabi claimed in his 1977 PhD thesis \cite{AJ} that all finitely presented $C(6)$-groups (except the obvious exceptions of cyclic and infinite dihedral groups) are SQ-universal. This claim has been restated in a still unpublished recent work by Al-Janabi, Collins, Edjvet and Spanu \cite{ACES}. Both results rely on elaborate proofs, involving many intricate constructions and case distinctions. 

In the first part of this paper, we give a concise proof that all \emph{infinitely presented} $C(6)$-groups are SQ-universal. This is the first direct proof of SQ-universality of infinitely presented small cancellation groups. Until very recently, no results on the SQ-universality of infinitely presented small cancellation groups were known. In \cite{GS} it is shown that the smaller class of infinitely presented $C(7)$-groups is acylindrically hyperbolic. Applying the theory of hyperbolically embedded subgroups \cite{DGO}, this provides a proof that infinitely presented $C(7)$-groups are SQ-universal. The result of \cite{GS} uses the fact that infinitely presented $C(7)$-groups are limits of hyperbolic groups. In contrast, infinitely presented $C(6)$-groups are not limits of hyperbolic groups in general, and, hence, require a different method of proof. The proof given in the present paper relies on the fact that $C(6)$-presentations are aspherical. It does not require any notion of hyperbolicity, and we 
expect that it can 
be generalized to further aspherical presentations. Our result also applies to many finitely presented $C(6)$-groups as explained in Remark~\ref{remark:finite_presentation}.

In the second part of the paper, we present a new viewpoint on graphical small cancellation presentations over free products that lets us extend our main result to groups defined by such presentations. Graphical small cancellation theory is a generalization of classical small cancellation theory. It was introduced by Gromov in the course of his construction of Gromov's monster \cite{Gro}. In a prior paper, we used graphical small cancellation theory to construct lacunary hyperbolic groups that coarsely contain prescribed infinite sequences of finite graphs \cite{Gru}. Classical and graphical small cancellation presentations over free products have provided various embedding theorems \cite[Chapter V]{LS} and, more recently, the first examples of torsion-free hyperbolic non-unique product groups \cite{Ste,AS}.

In the third part of the paper, we give the first examples of groups distinguishing the class of groups defined by \emph{metric} $C'(\lambda)$ small cancellation presentations from the class of groups defined by \emph{non-metric} $C(p)$-presentations. It is well-known that the class of $C(6)$-groups is strictly larger than the class of $C(7)$-groups, and therefore strictly larger than the class of $C'(\frac{1}{6})$-groups. This follows from the fact that there exist $C(6)$-groups containing $\Z^2$, while any $C(7)$-group cannot contain $\Z^2$. It is immediate from the definitions that the class of $C(p)$-groups contains the class of $C'(\frac{1}{p-1})$-groups. 

We show that for every $p>6$, the class of $C(p)$-groups is \emph{strictly} larger than the class of $C'(\frac{1}{p-1})$-groups by producing uncountably many pairwise non-isomorphic $C(p)$-groups that do not admit any $C'(\frac{1}{6})$-presentation. In fact, the $C(p)$-groups we construct have the stronger property that they do not admit any graphical $Gr'(\frac{1}{6})$-presentation. To prove this, we show a result of independent interest: In every, possibly infinitely presented, graphical $Gr'(\frac{1}{6})$-group, every cyclic subgroup is undistorted.



\subsection{Statement of results}
We briefly state our results here. Our proofs actually yield slightly stronger statements, see the referenced theorems.

\begin{defi} A group $G$ is \emph{SQ-universal} if for every countable group $C$ there exists a quotient $Q$ of $G$ such that $C$ embeds into $Q$. 
\end{defi}

\begin{thm}[cf. Theorem~\ref{thm:c6}]
 Let $G=\langle S\mid R\rangle$ be a $C(6)$-presentation, where $S$ is finite and $R$ is infinite. Then $G$ is SQ-universal.
\end{thm}

We discuss \emph{graphical} small cancellation conditions. For precise definitions, see Section~\ref{section:graphical}. Given a graph $\Gamma$ labelled by a set $S$, the group $G(\Gamma)$ is defined as the quotient of the free group on $S$ by the normal subgroup generated by all words read on closed paths in $\Gamma$. Given a free product $*_{i\in I}{G_i}$ and a graph labelled over $\sqcup_{i\in I} S_i$, where for each $i$, $S_i$ is a generating set of $G_i$, $G(\Gamma)$ is the quotient of $*_{i\in I}{G_i}$ by the normal subgroup generated by all words read on closed paths in $\Gamma$. 

A \emph{piece} in $\Gamma$ is a labelled path that occurs in two distinct places in $\Gamma$. The graphical versions of the $C(6)$-condition require that no nontrivial closed path is the concatenation of fewer than $6$ pieces. The graphical versions of the $C'(\frac{1}{6})$-condition require that any piece $p$ that is a subpath of a simple closed path $\gamma$ satisfies $|p|<\frac{|\gamma|}{6}$. The graphical $Gr$-conditions allow for label-preserving automorphisms of $\Gamma$, which corresponds to considering relators that are proper powers in the classical case.

\begin{thm}[cf. Theorem~\ref{thm:gr6}]
 Let $\Gamma$ be a $Gr_*(6)$-labelled graph over a free product of infinite groups that has at least 16 pairwise non-isomorphic finite components with nontrivial fundamental groups. Then $G(\Gamma)$ is SQ-universal.
\end{thm}

Given a finitely generated group $G$, a finitely generated subgroup $H\leqslant G$ is \emph{undistorted} if the inclusion map $H\to G$ is a quasi-isometric embedding with respect to the corresponding word-metrics.

\begin{thm}[cf. Theorem~\ref{thm:undistorted}]\label{thm:intro_undistorted}
 Suppose the set of labels is finite. Let $\Gamma$ be a $C'(\frac{1}{6})$-labelled graph, or let $\Gamma$ be a $Gr'(\frac{1}{6})$-labelled graph whose components are finite. Then every cyclic subgroup of $G(\Gamma)$ is undistorted.
\end{thm}

Theorem~\ref{thm:intro_undistorted} also extends to graphical small cancellation presentations over free products under additional assumptions on the free factors, see Theorem~\ref{thm:undistorted_product}.

Every group has a $Gr'(\frac{1}{6})$-presentation given by its labelled Cayley graph \cite[Example 2.2]{Gru}. Therefore, the restriction that components are finite in the $Gr'(\frac{1}{6})$-case of Theorem~\ref{thm:intro_undistorted} is necessary. If $\Gamma$ is finite, then $G(\Gamma)$ is hyperbolic \cite{Oll}, and the statement is classical. For infinite classical $C'(\frac{1}{6})$-presentations, the fact that cyclic subgroups are undistorted can be deduced from the facts that every infinitely presented classical $C'(\frac{1}{6})$-group acts properly on a CAT(0) cube complex \cite{AO} and that every group that acts properly on a CAT(0) cube complex has no distorted cyclic subgroups \cite{Hag}. Since there exist graphical $C'(\frac{1}{6})$-groups with property~(T) \cite{Gro,Sil}, an argument using CAT(0) cube complexes cannot be extended to graphical presentations.

\begin{thm}[cf. Theorem \ref{thm:uncountable}]\label{thm:intro_uncountable}
 Given $p\in\N$, there exist uncountably many pairwise non-quasi-isometric finitely generated groups $(G_i)_{i\in I}$ such that:
 \begin{itemize}
  \item Every $G_i$ admits a classical $C(p)$-presentation with a finite generating set.
  \item No $G_i$ is isomorphic to any group defined by a $C'(\frac{1}{6})$-labelled graph with a finite set of labels.
  \item No $G_i$ is isomorphic to any group defined by $Gr'(\frac{1}{6})$-labelled graph whose components are finite with a finite set of labels.
 \end{itemize}
\end{thm}
To prove Theorem~\ref{thm:intro_uncountable}, we construct infinitely presented classical $C(p)$-groups with distorted cyclic subgroups and apply Theorem~\ref{thm:intro_undistorted}. We obtain uncountably many examples using a construction of Bowditch \cite{Bow}.

%% file: 3classical4.tex
\section{SQ-universality of classical $C(6)$-groups}\label{section:classical}
In this section, we prove the SQ-universality of infinitely presented classical $C(6)$-groups using the following definition and result from \cite{Ols}:
\begin{defi}[\cite{Ols}] Let $G$ be a group and $F$ a subgroup. Then $F$ has the \emph{congruence extension property} (\emph{CEP}) if for every normal subgroup $N$ of $F$ (i.e. $N$ is normal in $F$), we have $\langle N\rangle ^G\cap F=N$, where $\langle N\rangle^G$ denotes the normal closure of $N$ in $G$. The group $G$ has \emph{property $F(2)$} if there exists a subgroup $F$ of $G$ that is a free group of rank $2$ and that has the CEP.
\end{defi}
\begin{prop}[\cite{Ols}] If a group $G$ has property $F(2)$, then $G$ is SQ-universal.
\end{prop}
We show:
\begin{thm}\label{thm:c6}
 Let $G=\langle S\mid R\rangle$ be a $C(6)$-presentation, where $S$ is finite and $R$ is infinite. Then $G$ has property $F(2)$.
\end{thm}
The proof of Theorem~\ref{thm:c6} in fact does not require a finite generating set and only a sufficient (finite) number of relators. The details of this are explained in Remark~\ref{remark:finite_presentation}.

\subsection{Classical small cancellation theory}
We begin with definitions and tools from classical small cancellation theory. The following definition of the classical $C(6)$ small cancellation condition is given in, for example, \cite[Chapter V]{LS}. Given a set $S$, $M(S)$ denotes the free monoid on $S\sqcup S^{-1}$. ``$\equiv$'' denotes letter-by-letter equality.
\begin{defi}
 Let $\langle S\mid R\rangle$ be a presentation, where $R\subseteq M(S)$. The set of relators $R$ is \emph{symmetrized} if all its elements are cyclically reduced and for every $r\in R$, all cyclic conjugates and their inverses of $r$ are in $R$.
 Given a symmetrized presentation $\langle S\mid R\rangle$, a \emph{piece} with respect to $\langle S\mid R\rangle$ is a word $u\in M(S)$ such that there exist $r\not\equiv r'$ in $R$ with $r\equiv uv$ and $r'\equiv uv'$ for words $v,v'$. We say that a presentation satisfies the \emph{$C(6)$-condition} if
 \begin{itemize}
  \item the set of relators is symmetrized and
  \item no relator is a product of fewer than 6 pieces.
 \end{itemize}
\end{defi}

As usual in small cancellation theory, we will translate problems about group presentations into diagrams. The notion of diagrams we use coincides with the standard notion of van Kampen diagrams as in \cite[Chapter V]{LS}, except that we allow faces whose boundary labels are not freely reduced. 

A \emph{diagram} $D$ is a finite, simply connected 2-complex with a fixed embedding into the plane or onto the 2-sphere such that the image of every 1-cell, called \emph{edge}, has an orientation and a label from a fixed set $S$. A \emph{singular disk diagram} is a diagram with a fixed embedding in the plane. A \emph{simple disk diagram} is a singular disk diagram homeomorphic to the 2-disk. The \emph{boundary} of a singular disk diagram $D$, denoted $\partial D$, is its topological boundary inside $\R^2$. A \emph{simple spherical diagram} is a diagram with a fixed homeomorphism onto the 2-sphere. A simple spherical diagram has empty boundary. 

Throughout this article, all \emph{paths} are edge-paths. Given a path $p$ in $D$, we denote by $\ell(p)$ its label in $M(S)$, the free monoid on $S\sqcup S^{-1}$, where the label of an edge $e$ is given exponent $+1$ if $e$ is traversed in its direction and exponent $-1$ if it is traversed in the opposite direction. We also denote by $\partial D$ the closed path traversing $\partial D$ (where we choose basepoint and orientation). The \emph{boundary label} of $D$ is the word read on $\partial D$. A \emph{face} $\Pi$ in a diagram is the image of a closed 2-cell $b$ under its characteristic map. The \emph{boundary cycle} $\partial \Pi$ is the image of the boundary cycle of $b$ (where we choose basepoint and orientation), and the \emph{boundary word} of $\Pi$ is the label of $\partial \Pi$.

An \emph{arc} in $D$ is a nontrivial path in $D$ whose interior vertices have degree 2 and whose initial and terminal vertex have degree different from 2. A \emph{spur} is an arc containing a vertex of degree 1. Given a path $p$, we denote by $\iota p$ its initial and by $\tau p$ its terminal vertex. An \emph{interior} edge is an edge not contained in $\partial D$. An interior face is a face whose boundary consists of interior edges. A face that is not interior is a \emph{boundary} face.

Given a presentation $\langle S\mid R\rangle$, where $R$ is a subset of $M(S)$, a \emph{diagram over} $\langle S\mid R\rangle$ is a diagram where every face has a boundary word in $R$. The following is a version of van Kampen's Lemma \cite[Section V.1]{LS}. Given $w\in M(S)$, a singular disk diagram \emph{for $w$} is a singular disk diagram whose boundary label is $w$.

\begin{thm} Let $G$ be given by the presentation $\langle S\mid R\rangle$. Then $w\in M(S)$ satisfies $w=1$ in $G$ if and only if there exists a singular disk diagram over $\langle S\mid R\rangle$ for $w$.
\end{thm}

The following is proven in \cite[Chapter 4, \S 11.6]{Ols89}. 

\begin{lem}\label{lem:reduction}
 Let $\langle S\mid R\rangle$ be a presentation, let $R'\subseteq R$, and let $D$ be a singular disk diagram over $\langle S\mid R\rangle$. Let $\Pi_1$ and $\Pi_2$ be faces in $D$ that share a vertex $v$. 
 \begin{itemize}
  \item If $\Pi_1=\Pi_2$, assume that, for some choice of orientation and basepoint, $\ell(\Pi_1)$ is freely equal to an element of the normal closure of $R'$ in $F(S)$. 
  \item If $\Pi_1\neq \Pi_2$, assume that $\ell(\Pi_1)\ell(\Pi_2)$, read from $v$ in counterclockwise direction, is freely equal to an element of the normal closure of $R'$ in $F(S)$.
 \end{itemize} 
 Then there exists a diagram $D'$ containing faces $f_1,f_2,\dots,f_k$, $k\geqslant 0$, such that each $f_i$ has a boundary word in $R'$, such that $D'$ has the same boundary word as $D$, and such that there exists an injection $f: \mathrm{faces}(D')\setminus\{f_1,f_2,\dots,f_k\} \to \mathrm{faces}(D)\setminus \{\Pi_1,\Pi_2\}$ such that the labels of $\partial\Pi$ and $f(\partial\Pi)$ coincide for each $\Pi\in\mathrm{faces}(D')\setminus\{f_1,f_2,\dots,f_k\}$.
\end{lem}

In the particular case that $R'=\emptyset$, the above lemma states that adjacent faces with freely inverse labels can be removed. 

We remark to the interested reader that the proof in \cite{Ols89} uses certain 0-faces (we do not use the terminology ``0-cells'' to avoid confusion), and the diagram $D'$ in \cite{Ols89} contains such 0-faces. These 0-faces are faces with labels $xsys^{-1}$, where $s\in S$ and $x$ and $y$ are powers of a symbol 1 that does not lie in $S\sqcup S^{-1}$ and that denotes the identity in $F(S)$. Using the following steps, which are best interpreted as operations on a planar graph, we can iteratively remove all 0-faces to obtain the statement above. Note that each of the operations does not alter the boundary word of $D'$ or of any $R$-face (unless the face is removed), and it does not alter the fact that $D'$ is a singular disk diagram.
\begin{itemize}
 \item Contract to a point an edge $e$ with $\ell(e)=1$ and $\iota e\neq \tau e$.
 \item Remove an edge $e$ with $\ell(e)=1$ and $\iota e=\tau e$, and also remove any subdiagram enclosed by $e$. (We contract $e$ and everything inside $e$ to a point.)
 \item Replace a face with label $ss^{-1}$ by an edge with label $s$. (We homotope one side of the bigon and the enclosed face onto the other side.) 
\end{itemize}

The following formula for curvature in spherical 2-complexes will be useful in our proofs. It is analogous to formulas proven in \cite[Section V.3]{LS}.

\begin{lem}[Curvature formula]\label{lem:curvature}
 Let $\Sigma$ be a 2-complex tessellating a 2-sphere. Then:
 \begin{equation*}
  6=\sum_{v\in\mathrm{0-cells}(\Sigma)}(3-d(v))+\frac{1}{2}\sum_{\Pi\in\mathrm{2-cells}(\Sigma)}(6-d(\Pi)),
 \end{equation*}
 where $d(v)$ denotes the degree of a 0-cell $v$ and $d(\Pi)$ denotes the length of $\partial \Pi$, i.e. the number 1-cells in the boundary of $\Pi$ (counted with multiplicity).
\end{lem}
\begin{proof}
 Let $V$ denote the number of 0-cells, $E$ the number of 1-cells and $F$ the number of 2-cells of $\Sigma$. Then it is well-known that $2=V-E+F.$ Moreover note:
 \begin{equation*}
  E=\frac{1}{2}\sum_{v\in\mathrm{0-cells}(\Sigma)}d(v)
   =\frac{1}{2}\sum_{\Pi\in\mathrm{2-cells}(\Sigma)} d(\Pi).
 \end{equation*}
Thus:
\begin{align*}
 6&=(3V-2E)+(3F-E)\\
  &=\sum_{v\in\mathrm{0-cells}(\Sigma)}(3-d(v))+\frac{1}{2}\sum_{\Pi\in\mathrm{2-cells}(\Sigma)}(6-d(\Pi)).
\end{align*}
\end{proof}

From now on, we fix a $C(6)$-presentation $\langle S\mid R\rangle$ for a group $G$ as in Theorem~\ref{thm:c6}. We will interpret $R$ as a \emph{labelled graph}:

\begin{defi}[Labelled graph]\label{defi:labelling} Let $\Gamma$ be a graph. A \emph{labelling} of $\Gamma$ over a set $S$ is a choice of orientation on each edge and map assigning to each edge an element of $S$. We call a graph with a labelling a \emph{labelled graph}.\end{defi}

The \emph{label} of a path in $\Gamma$ is defined in the same way as the label of a path in a diagram. A \emph{label-preserving} graph homomorphism is a graph homomorphism that preserves the labelling (i.e. both the orientation and the map to the set of labels).

For $r\in R$, let $[r]$ denote the class of all cyclic conjugates and their inverses of $r$. Each $[r]\subseteq R$ gives rise to a labelled cycle graph $\gamma_{[r]}$ of length $|r|$ that is labelled by the word $r$. For simplicity, we write $\gamma_r$ for $\gamma_{[r]}$. We denote $\Gamma_R:=\sqcup_{[r]\subseteq R}\gamma_r$. Thus we can consider subwords of relators as reduced paths in $\Gamma_R$. A piece then corresponds to a reduced labelled path $p$ which, considered as labelled line graph, admits two distinct label-preserving maps to $\Gamma_R$ such that there does not exist a label-preserving automorphism $\phi$ of $\Gamma_R$ making the following diagram commute:
\begin{center}
\begin{tikzpicture}
  \matrix (m) [matrix of math nodes,row sep=0em,column sep=2em,minimum width=2em] {
     \,& \Gamma_R \\
     p &\,\\
     \,& \Gamma_R \\};
  \path[-stealth]
    (m-2-1) edge (m-1-2)
            edge (m-3-2)
    (m-1-2) edge node [right] {$\phi$} (m-3-2);
\end{tikzpicture}
\end{center}

\subsection{Proof of Theorem~\ref{thm:c6}} The strategy of proof of Theorem~\ref{thm:c6} is the following: We define group elements $\alpha_1,\alpha_2$ as suitable products of subwords of relators. Then we prove that $\alpha_1$ and $\alpha_2$ freely generate a free subgroup $F$ of $G$ that has the CEP. This is achieved by translating the problem into spherical diagrams and applying the curvature formula (Lemma \ref{lem:curvature}). The $\alpha_i$ will be products of ``halves of relators'' in the sense of piece distance (Definition \ref{defi:piece_distance}), as defined in \cite{Gru}. We use the notion of \emph{support} (see Definition~\ref{defi:support}) to ensure that no free cancellation occurs when forming the products.

\begin{defi}[Piece distance]\label{defi:piece_distance} Let $r\in R$, and let $x$ and $y$ be vertices in $\gamma_r$. The \emph{piece distance} of $x$ and $y$, denoted $d_p(x,y)$, is the least number of pieces whose concatenation is a path in $\gamma_r$ from $x$ to $y$. If there is no such path, set $d_p(x,y)=\infty$. 
\end{defi}

\begin{defi}[Support of a vertex]\label{defi:support} Let $r\in R$, and let $v$ be a vertex in $\gamma_r$. The \emph{support of $v$}, denoted $\supp(v)$, is the set of labels of paths of length 1 starting at $v$. 
\end{defi}

The elements of $\supp(v)$ are in $S\sqcup S^{-1}$, and $|\supp(v)|=2$ since words in $R$ are cyclically reduced. Since $S$ is finite and $R$ is infinite, we may choose an infinite subset $R_0\subseteq R$ such that each $r\in R_0$ is a product of pieces. 

The following is immediate from the $C(6)$-condition. For a proof, see \cite[Lemma~3.6]{Gru}.
\begin{lem}\label{lem:piece_distance} Let $r\in R_0$, and let $x$ be a vertex in $\gamma_r$. Then there exists a vertex $y\in \gamma_r$ with $d_p(x,y)\geqslant 3$. 
\end{lem}

\begin{lem}\label{lem:find} There exist relators $r_1,...,r_{16}$ in $R_0$ and vertices $x_n,y_n\in \gamma_{r_n}$ with:
\begin{itemize}
 \item $[r_n]=[r_m]\Leftrightarrow n=m$,
 \item $d_p(x_n,y_n)\geqslant 3$,
 \item $\supp(y_n)\cap\supp(x_{n+1})=\emptyset$ for $n\in\{1,...,16\}\setminus\{8,16\}$.
\end{itemize}
\end{lem}

The following proof aims to use the least number of relators possible in the construction. This will be used in Remark~\ref{remark:finite_presentation} for the case that $R_0$ is finite. If, as assumed now, $R_0$ is infinite, many technicalities, such as keeping track of $L$ and $L'$, can be skipped.

\begin{proof}[Proof of Lemma~\ref{lem:find}.]

Take $R_1$ to be a set of representatives of $[\cdot]$-classes in $R_0$.
By Lemma \ref{lem:piece_distance}, for every $r\in R_1$, for every vertex $x\in \gamma_r$ there exists a vertex $y\in \gamma_r$ with $d_p(x,y)\geqslant 3$.
We construct the $r_k$ in $R_1$ inductively: Given $r_k$ and $x_k$, we pick $y_k$ with $d_p(x_k,y_k)\geqslant 3$ such that there exists $r_{k+1}$ distinct from all $r_i$, $i\leqslant k$, and a vertex $x_{k+1}$ in $\gamma_{k+1}$ such that $\supp(y_k)\cap \supp(x_{k+1})=\emptyset$. If this is possible for every $k<16$, the claimed sequence exists.
 
Now suppose there exists some $K<16$ such that for every choice of $y_K$ with $d_p(x_K,y_K)\geqslant3$, every vertex $x$ in every $\gamma_r$ with $r\in R_2:=R_1\setminus\{r_1,r_2,\dots,r_K\}$ satisfies $\supp(y_K)\cap \supp(x)\neq\emptyset$. Choose $y_K$ with $d_p(x_K,y_K)\geqslant 3$ in $\gamma_{r_K}$. 
There are two cases to consider:
\begin{itemize}
 \item[1)] $\supp(y_K)=\{a^{-1},b\}$ for $a\neq b$, $a,b\in S\sqcup S^{-1}$. Then every $r\in R_2$ has the form (up to inversion and cyclic conjugation) $r=a^{k_1}b^{-k_2}a^{k_3}...b^{-k_l}$ where all $k_i>0$.
 \item[2)] $\supp(y_K)=\{a^{-1},a\}$ for $a\in S$. Then every $r\in R_2$ has the form (up to cyclic conjugation) $r=a^{k_1}s_1a^{k_2}s_2a^{k_3}...s_l$, where $k_i\in \Z\setminus\{0\}$, $s_i\in S\sqcup S^{-1}$.
\end{itemize}

If $K<8$, set $L:=0$. If $K\geqslant 8$, set $L:=8$. We continue by choosing anew the relators $r_{L+1},...,r_{16}$ in $R_2$ (keeping the chosen $r_1,...,r_8$ and $x_1,...,x_8$ and $y_1,...,y_8$ if $L=8$). 

Suppose we are in case 1. On all but at most two elements of $R_2$, all subwords of the form $a^l$ or $b^l$ (called $a$-\emph{blocks} respectively $b$-\emph{blocks}) are pieces. The only (up to two) relators where this may not be the case are relators where maximal powers of $a$, respectively $b$, occur. Let $R_3$ denote the subset of $R_2$ where all $a$-blocks and all $b$-blocks are pieces.

Choose $r_{L+1}$ arbitrary in $R_3$ and $x_{L+1}$ arbitrary in $\gamma_{r_{L+1}}$. We claim: There exists $y_{L+1}$ in $\gamma_{r_{L+1}}$ with $d_p(x_{L+1},y_{L+1})\geqslant 3$ such that $y_{L+1}$ lies in the intersection of an $a$-block and a $b$-block. By Lemma \ref{lem:piece_distance}, there exists a vertex $y$ in $\gamma_{r_{L+1}}$ with $d_p(x_{L+1},y)\geqslant 3$. Suppose it lies inside an $a$-block or $b$-block $\beta$. Suppose both endpoints of $\beta$, denoted $\iota\beta$ and $\tau\beta$, have piece-distance $\leqslant2$ from $x_{L+2}$. Then there are edge-disjoint paths $x_{L+1}\to\iota\beta$ and $x_{L+1}\to\tau\beta$ on $\gamma_{r_{L+1}}$ each consisting of at most $2$ pieces. Since $\beta$ is a piece, $r_{L+1}$ is made up of at most 5 pieces, a contradiction. Thus we can choose $y_{L+1}$ as claimed. Since we are in case 1, we can now choose any $r_{L+2}\in R_3\setminus \{r_{L+1}\}$, and there exists $x_{L+2}$ in $\gamma_{r_{L+2}}$ with $\supp(y_{L+1})\cap\supp(x_{L+2})=\emptyset$. Since $r_{L+1}$ 
and $x_{L+1}$ were arbitrary, we can proceed inductively to complete the proof.

Suppose we are in case 2. Then $y_K$ is in the interior of an $a$-block $\beta$ (which may be the entire relator $r_K$). If $\beta$ is a piece, we can use the same argument as above to replace $y_K$ by a boundary vertex of $\beta$, reducing the problem to case 1. If $\beta$ is not a piece, then it is the maximal power of $a$ occurring in any relator, and all $a$-blocks occurring in other relators are pieces.

By the same argument as above, whenever we have a vertex $x$ in $\gamma_r$ for $r\in R_2$, we can find a vertex $y$ with $d_p(x,y)\geqslant 3$ that is endpoint of an $a$-block. Now we apply our initial naive algorithm to the set $R_2$, i.e. inductively try to construct the set of relators $r_{L+1},...,r_{16}$ and corresponding vertices. If this algorithm fails to choose $x_{K'+1}$ on a new relator, where $K'<16$, we choose $y_{K'}$ with $d_p(x_{K'},y_{K'})\geqslant 3$ as endpoint of an $a$-block. Since the algorithm fails, we are in case 1 (for the set of relators $R_3:=R_2\setminus \{r_{L+1},...,r_{L+K'}\}$) with the additional property that all $a$-blocks are pieces (because we have already excluded maximal powers of $a$ by excluding $r_K$), and all $b$-blocks are pieces (because they have length 1).

If $K'<8$, let $L':=0$. If $K'\geqslant 8$, let $L':=8$. We choose anew the relators $r_{L+L'+1},...,r_{16}$ and corresponding vertices as in case 1, not having to exclude the maximal powers of $a$ and $b$ by the additional property that $a$-blocks and $b$-blocks are pieces. 
\end{proof}

\begin{remark}\label{remark:finite_presentation} We show that Lemma~\ref{lem:find} also applies large enough \emph{finite} presentations. Proposition~\ref{prop:main_classical} will show that also in these cases, the defined group has property $F(2)$.

Let $\langle X\mid Y\rangle$ be a presentation. For each $r\in Y$, denote by $[r]$ the set of all cyclic conjugates of $r$ and their inverses. A \emph{concise refinement} is a presentation $\langle X\mid \tilde Y\rangle$, where $\tilde Y$ is a set of representatives for the $[\cdot]$-classes in $Y$ (i.e. for each $[r]\in \{[\rho]:\rho\in Y\}$, we choose exactly one element of $[r]$). 

A \emph{Tietze-reduction} of a presentation is obtained as follows: We call a relator $r$ \emph{redundant in} $\langle X\mid Y\rangle$ if it contains a subword $s^\epsilon$, $s\in X,\epsilon\in\{\pm1\}$, where $s^\epsilon$ occurs exactly once in $r$ and in no other relator, and $s^{-\epsilon}$ occurs in no relator. We call the generator $s$ \emph{redundant} in $\langle X\mid Y\rangle$ if it occurs in such a way in a redundant relator. The Tietze-reduction of $\langle X\mid Y\rangle$ is obtained as follows: Simultaneously remove from $Y$ all redundant relators, and for each of the reduntant relators $r$, remove from $X$ one redundant generator in $r$. 
This operation is a Tietze-transformation, i.e. the resulting presentation defines the same group. (This is not an iterative process, i.e. the Tietze-reduction of a presentation may again contain redundant relators and generators.) 

Suppose $\langle X\mid Y\rangle$ is symmetrized. Note that if a relator $r$ in the Tietze reduction $\langle X'\mid Y'\rangle$ of a concise refinement of $\langle X\mid Y\rangle$ is not a product of pieces with respect to $Y$, then it is a proper power, and for any vertex $x$ in $\gamma_r$ there exists $y$ in $\gamma_r$ with $d_p(x,y)=\infty$. Therefore, we can apply Lemma~\ref{lem:find} to $\langle X'\mid Y'\rangle$. Analyzing the proof shows that if $|Y'|\geqslant 30$, then the conclusion of the lemma holds. 
\end{remark}

\begin{remark}\label{remark:no_simple_groups} By \cite{AJ,ACES}, any group given by a $C(6)$-presentation $\langle X\mid  Y\rangle$ with $|Y|<\infty$ is either cyclic, infinite dihedral, or SQ-universal. (The generating set $X$ may be infinite.) Thus, by Remark~\ref{remark:finite_presentation}, \emph{every} group given by a $C(6)$-presentation (with no restrictions on the cardinalites of the sets of generators and relators) is either cyclic, infinite dihedral, or SQ-universal. By \cite{BW}, no $C(6)$-group can contain $F_2\times F_2$ as a subgroup. Thus, every infinite $C(6)$-group must have a nontrivial proper quotient, i.e. there does not exist an infinite simple $C(6)$-group.
\end{remark}

\noindent {\bf Definition of the free subgroup with the CEP.} For each $k\in\{1,...,16\}$, we denote $\gamma_k:=\gamma_{r_k}$. Let $\alpha_1$ and $\alpha_2$ be symbols not in $S$, and denote for $i\in\{1,2\}$:
$$W_i:=\{\alpha_i^{-1}\ell(p_{8i-7})\ell(p_{8i-6})\dots\ell(p_{8i})\mid p_k\text{ a simple path from }x_k\text{ to }y_k\}.$$

Note that for each $p_k$, there are exactly two choices, i.e. $|W_i|=2^8$. Let $\alpha:=\{\alpha_1,\alpha_2\}$ and $W:=W_1\cup W_2$. As groups, we have $\langle S\mid R\rangle\cong \langle S,\alpha\mid W,R\rangle$. Using the presentation on the right-hand side, $\alpha$ admits a map to $G$. We claim that the image of $\alpha$ in $G$ generates a rank 2 free subgroup with the CEP. Showing this claim will complete our proof of Theorem~\ref{thm:c6}.
 
We call a subword of an element of $W$ that is one of the $\ell(p_k)$ a \emph{block}. The blocks adjacent to $\alpha_i$ (where the word is considered cyclic) are called \emph{boundary blocks}, all other blocks are called \emph{interior blocks}. Each block is identified with a path in $\gamma_k$ for some $k$. We call this identification \emph{lift}. If $\Pi$ is a face in a diagram and the label of $\Pi$ lies in $R$, then $\partial \Pi$ admits a label-preserving map to $\Gamma_R$. We call any such map a \emph{lift}.

\begin{prop}\label{prop:main_classical}
The set $\alpha$ injects into $G$, the image of $\alpha$ freely generates a free subgroup $F$ of $G$, and $F$ has the CEP.
\end{prop}
\begin{proof} 
We first show the CEP. Let $N\leqslant F$ be normal in $F$. Suppose there exists $g\in (\langle N\rangle ^G\cap F)\setminus  N$. Let $L$ be the set of words in $\alpha\sqcup\alpha^{-1}$ representing elements of $N$, and consider the presentation $\langle S,\alpha\mid L,W,R\rangle$. Then there exists a word $w$ in $\alpha\sqcup\alpha^{-1}$ representing $g$ and a diagram $D$ over $\langle S,\alpha\mid L,W,R\rangle$ for $w$. Assume $g$, $w$, and $D$ are chosen such that the $(L,W,R)$-lexicographic area of $D$ is minimal for all possible choices (i.e. we first minimize the number of faces labelled by elements of $L$, then the number of faces labelled by elements of $W$ and then the number of faces labelled by elements of $R$), and among these choices, the number of edges of $D$ is minimal. Note that by assumption, $w$ is freely nontrivial, i.e. $D$ has at least one face. We will construct from $D$ a 2-complex $\Sigma'$ tessellating a 2-sphere that violates the curvature formula (Lemma \ref{lem:curvature}), 
whence $w$ does not exist and our claim holds.

\vspace{10pt}

\noindent{\bf Claim 1.} $D$ has the following properties:
\begin{itemize}
 \item [a)] $D$ is a simple disk diagram, and $w$ is cyclically reduced.
 \item [b)] No $L$-face intersects $\partial D$. Therefore, every edge of $\partial D$ is contained in a $W$-face.
 \item [c)] Every $L$-face is simply connected, and no two $L$-faces intersect. Therefore, every $L$-face shares all its boundary edges with $W$-faces. We say it is \emph {surrounded} by $W$-faces.
 \item [d)] The intersection of two $W$-faces does not contain an $\alpha$-edge. For a path $a$ in the intersection of two blocks $\beta$ and $\beta'$ of two $W$-faces, the lifts $a\to\beta\to \Gamma_R$ and $a\to\beta'\to \Gamma_R$ do not coincide.  
 \item [e)] A path in the intersection of two $R$-faces is a piece. For a path $a$ in the intersection of a block $\beta$ of a $W$-face and an $R$-face $\Pi$ the lift $a\to\beta\to\Gamma_R$ does not coincide with any lift $a\to\partial\Pi\to\Gamma_R$ and, therefore, is a piece.
 \item [f)] Every $R$-face is simply connected.
 
\end{itemize}

a) $\partial D$ is a product of conjugates of the boundary labels of its simple disk components. At least one of these components must have label not in $L$, and, by minimality, equals $D$. If the boundary word of $D$ is not cyclically reduced, we can fold together inverse edges in its boundary as in Figure~\ref{figure:folding} to reduce the number of edges of $D$, contradicting minimality. Thus, $w$ is cyclically reduced.
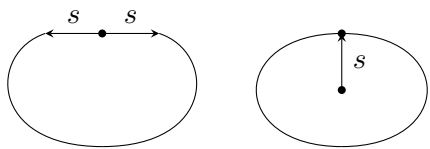
\begin{figure}\label{figure:folding}
 \begin{tikzpicture}[>=stealth,scale=.5]
\draw[-] (1.5,0) .. controls (3,-.5) and (3,-3) .. (0,-3);
\draw[-] (-1.5,0) .. controls (-3,-.5) and (-3,-3) .. (0,-3);
\fill (0,0) circle (3pt);
\draw [->] (0,0) to node [above] {\small $s$} (1.5,0);
\draw [->] (0,0) to node [above] {\small $s$} (-1.5,0);
\end{tikzpicture}
\begin{tikzpicture}[>=stealth,scale=.5]
\draw[-] (0,0) .. controls (3,0) and (3,-3) .. (0,-3);
\draw[-] (0,0) .. controls (-3,0) and (-3,-3) .. (0,-3);
\fill (0,0) circle (3pt);
\fill (0,-1.5) circle (3pt);
\draw [->] (0,-1.5) to node [right] {\small $s$} (0,0);
\end{tikzpicture}
\caption{Folding together inverse edges $e_1$ and $e_2$ with the same label $s$.}
\end{figure}

b) Suppose an $L$-face $\Pi$ intersects $\partial D$ in a vertex $v$. Let $h$ be the initial subpath of $\partial D$ terminating at $v$, and $h'$ the terminal subpath of $\partial D$ starting at $v$. 
Consider $\partial \Pi$ with basepoint $v$ and with the same orientation as $\partial\Pi$. Then we have (in $F(S\cup\alpha)$) $w':=\ell(h)\ell(\partial \Pi)^{-1}\ell(h')=\ell(\partial D)\ell(h')^{-1}\ell(\partial \Pi)^{-1}\ell(h')$, whence $w'$ represents an element of $(\langle N \rangle^G\cap F) \setminus N$. We can remove $\Pi$ and ``cut up'' the resulting annulus as in Figure \ref{figure:boundary_face} to obtain a diagram for $w'$ that has fewer $L$-faces than $D$, contradicting minimality.
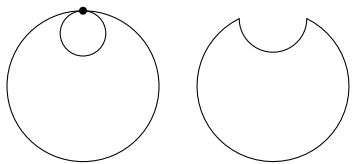
\begin{figure}\label{figure:boundary_face}
\begin{tikzpicture}[line cap=round,line join=round,x=1cm,y=1cm,scale=.5]
\draw(2,3) circle (2cm);
\draw(2,4.4) circle (0.6cm);
\draw [shift={(7,3)}] plot[domain=-1.57:1.11,variable=\t]({1*2*cos(\t r)+0*2*sin(\t r)},{0*2*cos(\t r)+1*2*sin(\t r)});
\draw [shift={(7,3)}] plot[domain=-1.57:1.11,variable=\t]({-1*2*cos(\t r)+0*2*sin(\t r)},{0*2*cos(\t r)+1*2*sin(\t r)});
\draw [shift={(7,4.79)}] plot[domain=-3.14:0,variable=\t]({1*0.89*cos(\t r)+0*0.89*sin(\t r)},{0*0.89*cos(\t r)+1*0.89*sin(\t r)});
\begin{scriptsize}
\fill [color=black] (2,5) circle (3pt);
\end{scriptsize}
\end{tikzpicture}
 \caption{If $\Pi$ intersects the boundary in a vertex $v$, we remove $\Pi$ and ``cut up'' the resulting annulus to obtain a singular disk diagram.}
\end{figure}

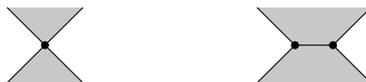
\begin{figure}\label{figure:blowup}
\begin{tikzpicture}[scale=.5]
\fill[mygrey] (0,1) to (1,0) to (2,1);
\fill[mygrey] (0,-1) to (1,0) to (2,-1);
\draw[-] (0,1) to (1,0) to (2,1);
\draw[-] (0,-1) to (1,0) to (2,-1);
\fill (1,0) circle (3pt);
\end{tikzpicture}
\hspace{2cm}
\begin{tikzpicture}[scale=.5]
\fill[mygrey] (0,1) to (1,0) to (2,0) to (3,1);
\fill[mygrey] (0,-1) to (1,0) to (2,0) to (3,-1);
\draw[-] (0,1) to (1,0);\draw[-] (2,0) to (3,1);
\draw[-] (0,-1) to (1,0);\draw[-] (2,0) to (3,-1);
\draw[-] (1,0) to (2,0);
\fill (1,0) circle (3pt);
\fill (2,0) circle (3pt);
\end{tikzpicture}
\caption{Blowing up a vertex in the intersection of two (grey) faces by inserting an edge.}
\end{figure}
c) If two distinct $L$-faces intersect in a vertex and if their labels are not freely inverse, we can merge them into one $L$-face because $N$ is normal in $F$. (Formally, we blow up the vertex in their intersection as in Figure~\ref{figure:blowup} and then remove the inserted edge.) This contradicts minimality. If they are inverse, we use Lemma~\ref{lem:reduction} to remove them, again contradicting minimality. If an $L$-face $\Pi$ is non-simply connected, it encloses some subdiagram $\Delta$. Then $\ell(\partial\Delta)\in L$ by the minimality assumption, and $\Pi$ and $\Delta$ can be merged into one $L$-face; this contradicts minimality. 

d) If two $W$-faces $\Pi$ and $\Pi'$ intersect in their $\alpha$-edges, then, by construction, $\ell(\Pi)\ell(\Pi')$ (read from a vertex in the intersection and with counterclockwise orientation) lies in the normal subgroup of $F(S\cup\alpha)$ generated by $R$. Thus, by Lemma~\ref{lem:reduction}, we can replace $\Pi$ and $\Pi'$ by $R$-faces, contradicting $(L,W,R)$-minimality.

Let $a$ be a path in the intersection of two blocks $\beta$ and $\beta'$ of two $W$-faces $\Pi$ and $\Pi'$. Suppose the maps $a\to\beta\to \Gamma_R$ and $a\to\beta'\to \Gamma_R$ coincide. Then $\ell(\Pi)\ell(\Pi')$, read from the initial vertex of the $\alpha_i$-edge of $\Pi$ reduces to a word:
$\alpha_iw\alpha_i^{-1}w'$ where $w$ and $w'$ lie in the normal subgroup of $F(S)$ generated by $R$. Thus, $\ell(\Pi)\ell(\Pi')$ lies in the normal subgroup of $F(S\cup\alpha)$ generated by $R$, and we can replace $\Pi$ and $\Pi'$ by $R$-faces, contradicting $(L,W,R)$-minimality.

e) An arc in the intersection of two $R$-faces is a piece, for otherwise, the faces are inverse and can be removed, contradicting minimality. Suppose for a path $a$ in the intersection of a block $\beta$ of a $W$-face $\Pi$ and an $R$-face $\Pi'$, the lift $a\to\beta\to\Gamma_R$ and a lift $a\to\partial\Pi'\to\Gamma_R$ coincide. We can glue onto $\beta$ an $R$-face $\Pi''$ whose label is inverse to that of $\Pi'$. This gluing occurs inside $\Pi$, i.e. we replace the $W$-face $\Pi$ by the union of $\Pi''$ and a new $W$-face $\tilde\Pi$ such that the exterior boundary of $\tilde\Pi\cup\Pi''$ equals the boundary of $\Pi$. Then we can remove $\Pi'$ and $\Pi''$ by Lemma~\ref{lem:reduction}, thus in total reducing the $(L,W,R)$-area, a contradiction.

f) For a contradiction, assume $\Pi$ is an innermost non-simply connected $R$-face, i.e. $\Pi$ encloses some simple disk diagram $\Delta$ in which every $R$-face is simply connected. Consider $\Delta$ on its own and glue on a face with label $\ell(\partial\Delta)$ to obtain a simple spherical diagram $\Delta'$. The proof of claim 2 will show that $\Delta'$ violates the curvature formula (Lemma~\ref{lem:curvature}) using the fact that every $R$-face of $\Delta$ is simply connected. We recommend first considering the proof of claim 2 and then going back to the following paragraph.

We only need to make one adaption in the proof of claim 2 when considering $\Delta'$: When deleting vertices of degree 2 in $\Delta'$, we do not delete the vertex $v$ of (possibly) degree 2 that, in $D$, lies in the intersection of $\Delta$ and $\Pi$ and is traversed twice in $\partial \Pi$. This will be necessary because an arc in $\partial \Delta$ in the diagram $\Delta$ containing $v$ in its interior may not be a piece. All other arcs in $\partial \Delta$ are pieces by claim 1e). The resulting curvature for $\Delta'$ is at most $(3-2)+\frac{1}{2} (6-1)<6$: The contribution $(3-2)$ comes from the possible degree 2 vertex, and the contribution $\frac{1}{2} (6-1)$ comes from the degree at least 1 face that we glued on.

\vspace{10pt}

\noindent{\bf Claim 2.} The existence of $D$ contradicts the curvature formula.

We construct a spherical diagram out of $D$: First we glue a new face with boundary label $w$ onto $D$ to obtain a simple spherical diagram $\Sigma$.

In $\Sigma$, each face whose boundary is a word in $\alpha\sqcup\alpha^{-1}$ is surrounded by $W$-faces by claims 1a) and 1b). For each $\alpha$-face $\Pi$, we add in a new vertex in the interior of $\Pi$, the \emph{apex}. Then we remove all boundary edges of $\Pi$ to obtain $\Pi'$, a face whose boundary is made up of the blocks that were contained in the boundary of the $W$-faces sharing edges with $\Pi$. Now we connect each vertex that lies at the end of a block to the apex by gluing in an edge, a so-called \emph{cone-edge}. (See Figure \ref{figure:wheel}.) 
We call the subdiagram of faces incident at the apex a \emph{wheel}, and each face in the wheel a \emph{cone}. Each cone has a boundary made up of two consecutive cone-edges and a block.

\begin{figure}\label{figure:wheel}
\begin{tikzpicture}[line cap=round,line join=round,x=.3cm,y=.3cm,scale=.5]
\draw (0,-5)-- (0,-8);
\draw (0,-8)-- (4.59,-6.55);
\draw (4.59,-6.55)-- (3.54,-3.54);
\draw (3.54,-3.54)-- (6.55,-4.59);
\draw (6.55,-4.59)-- (7.88,-1.39);
\draw (7.88,-1.39)-- (5,0);
\draw (5,0)-- (7.88,1.39);
\draw (7.88,1.39)-- (6.55,4.59);
\draw (6.55,4.59)-- (3.54,3.54);
\draw (3.54,3.54)-- (4.59,6.55);
\draw (4.59,6.55)-- (1.39,7.88);
\draw (1.39,7.88)-- (0,5);
\draw (0,5)-- (-1.39,7.88);
\draw (-1.39,7.88)-- (-4.59,6.55);
\draw (-4.59,6.55)-- (-3.54,3.54);
\draw (-3.54,3.54)-- (-6.55,4.59);
\draw (-6.55,4.59)-- (-7.88,1.39);
\draw (-7.88,1.39)-- (-5,0);
\draw (-5,0)-- (-7.88,-1.39);
\draw (-7.88,-1.39)-- (-6.55,-4.59);
\draw (-6.55,-4.59)-- (-3.54,-3.54);
\draw (-3.54,-3.54)-- (-4.59,-6.55);
\draw (-4.59,-6.55)-- (0,-8);
\draw(0,-2) circle (.9cm);
\fill   (0,-5) circle (1.5pt);
\fill   (3.54,-3.54) circle (1.5pt);
\fill   (5,0) circle (1.5pt);
\fill   (3.54,3.54) circle (1.5pt);
\fill   (0,5) circle (1.5pt);
\fill   (-3.54,3.54) circle (1.5pt);
\fill   (-5,0) circle (1.5pt);
\fill   (-3.54,-3.54) circle (1.5pt);
\fill   (0,-8) circle (1.5pt);
\fill   (4.59,-6.55) circle (1.5pt);
\fill   (6.55,-4.59) circle (1.5pt);
\fill   (7.88,-1.39) circle (1.5pt);
\fill   (7.88,1.39) circle (1.5pt);
\fill   (6.55,4.59) circle (1.5pt);
\fill   (4.59,6.55) circle (1.5pt);
\fill   (1.39,7.88) circle (1.5pt);
\fill   (-1.39,7.88) circle (1.5pt);
\fill   (-4.59,6.55) circle (1.5pt);
\fill   (-6.55,4.59) circle (1.5pt);
\fill   (-7.88,1.39) circle (1.5pt);
\fill   (-7.88,-1.39) circle (1.5pt);
\fill   (-6.55,-4.59) circle (1.5pt);
\fill   (-4.59,-6.55) circle (1.5pt);
\end{tikzpicture}
\hspace{2cm}
\begin{tikzpicture}[line cap=round,line join=round,x=.3cm,y=.3cm,scale=.5]
\draw (0,-5)-- (0,-8);
\draw (0,-8)-- (4.59,-6.55);
\draw (4.59,-6.55)-- (3.54,-3.54);
\draw (3.54,-3.54)-- (6.55,-4.59);
\draw (6.55,-4.59)-- (7.88,-1.39);
\draw (7.88,-1.39)-- (5,0);
\draw (5,0)-- (7.88,1.39);
\draw (7.88,1.39)-- (6.55,4.59);
\draw (6.55,4.59)-- (3.54,3.54);
\draw (3.54,3.54)-- (4.59,6.55);
\draw (4.59,6.55)-- (1.39,7.88);
\draw (1.39,7.88)-- (0,5);
\draw (0,5)-- (-1.39,7.88);
\draw (-1.39,7.88)-- (-4.59,6.55);
\draw (-4.59,6.55)-- (-3.54,3.54);
\draw (-3.54,3.54)-- (-6.55,4.59);
\draw (-6.55,4.59)-- (-7.88,1.39);
\draw (-7.88,1.39)-- (-5,0);
\draw (-5,0)-- (-7.88,-1.39);
\draw (-7.88,-1.39)-- (-6.55,-4.59);
\draw (-6.55,-4.59)-- (-3.54,-3.54);
\draw (-3.54,-3.54)-- (-4.59,-6.55);
\draw (-4.59,-6.55)-- (0,-8);
\draw (-3.54,-3.54)-- (3.54,3.54);
\draw (-5,0)-- (5,0);
\draw (-3.54,3.54)-- (3.54,-3.54);
\draw (0,5)-- (0,-5);
\fill  (0,-5) circle (1.5pt);
\fill   (3.54,-3.54) circle (1.5pt);
\fill   (5,0) circle (1.5pt);
\fill   (3.54,3.54) circle (1.5pt);
\fill   (0,5) circle (1.5pt);
\fill   (-3.54,3.54) circle (1.5pt);
\fill   (-5,0) circle (1.5pt);
\fill   (-3.54,-3.54) circle (1.5pt);
\fill   (0,-8) circle (1.5pt);
\fill   (4.59,-6.55) circle (1.5pt);
\fill   (6.55,-4.59) circle (1.5pt);
\fill   (7.88,-1.39) circle (1.5pt);
\fill   (7.88,1.39) circle (1.5pt);
\fill   (6.55,4.59) circle (1.5pt);
\fill   (4.59,6.55) circle (1.5pt);
\fill   (1.39,7.88) circle (1.5pt);
\fill   (-1.39,7.88) circle (1.5pt);
\fill   (-4.59,6.55) circle (1.5pt);
\fill   (-6.55,4.59) circle (1.5pt);
\fill   (-7.88,1.39) circle (1.5pt);
\fill   (-7.88,-1.39) circle (1.5pt);
\fill   (-6.55,-4.59) circle (1.5pt);
\fill   (-4.59,-6.55) circle (1.5pt);
\fill   (0,0) circle (1.5pt);
\end{tikzpicture}
\caption{Left: An $\alpha$-face corresponding to a word in $\alpha\sqcup\alpha^{-1}$ of length 1 is surrounded by 1 $W$-face. The exterior boundary of the $W$-face decomposes into 8 blocks. Right: The $\alpha$-face and $W$-face are replaced by a wheel. The vertex in the wheel coming from the intersection of the two boundary blocks may have degree 2. The vertices coming from the endpoints of interior blocks must have degree at least 3, since they come from gluing together vertices with disjoint support.}
 \end{figure}
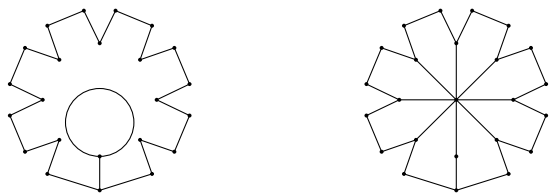

Suppose there are cones $\Pi_1$ and $\Pi_2$ with corresponding blocks $\beta_1$ and $\beta_2$ such that an arc $a$ in $\beta_1\cap\beta_2$ is not a piece, and suppose $\Pi_1\neq\Pi_2$. Then both $\beta_1$ and $\beta_2$ lift to the same $\gamma_r$. We now remove the arc $a$, turning $\Pi_1$ and $\Pi_2$ into a new face $\Pi$. The labelled subpaths of the boundary of $\Pi$ admit label-preserving maps to $\gamma_r$, i.e. they are labelled by subwords of a relator. (To be more precise, every reduced subword is a subword of a relator.) Hence, by Claim 1d), any arc in the intersection of an $R$-face with $\Pi$ is a piece. By Claim 1e), $\Pi$ has no more than two consecutive cone-edges, i.e. any adjacent pair of cone-edges in $\Pi$ is separated from any other pair by non-empty paths.

We call a face that arises from merging multiple cones a \emph{star}. We iterate the above procedure as follows: Whenever an arc in the intersection of two distinct cones, or in the intersection of a cone and a star is not a piece, remove that arc. Note that, by claims 1d) and 1e), still, an arc in the intersection of an $R$-face with a star is a piece, and any adjacent pair of cone-edges in a star is separated from any other pair by non-trivial paths.
In the end, the resulting spherical diagram has the following properties:
\begin{itemize}
 \item Each arc in the boundary of an $R$-face is a piece, i.e. its boundary consists of no fewer than 6 arcs.
 \item Each cone has a boundary made up of 2 cone-edges and a block $\beta$. If $\beta$ does not self-intersect nontrivially in an arc, then every arc contained in $\beta$ is a piece, i.e. $\beta$ consists of no fewer than 3 arcs. If $\beta$ does self-intersect in an arc, then it consists of no fewer than 3 arcs in any case, since we are counting arcs with multiplicity.
\end{itemize}

By our assumptions on the supports of vertices, every vertex that is endpoint of an interior block has degree at least 3. We deduce for the boundaries of cones:
\begin{itemize}
 \item A cone coming from an interior block has at least 5 arcs.
 \item A star coming from interior blocks has at least 6 arcs.
 \item A cone coming from a boundary block has at least 4 arcs.
 \item A star coming from boundary blocks has at least 4 arcs.
\end{itemize}

We now iteratively remove all vertices of degree 2, thus replacing each arc by a single edge. Denote the resulting spherical 2-complex by $\Sigma'$ and consider the curvature formula (Lemma \ref{lem:curvature}). 

All (images of) $R$-faces have degree at least $6$ by the $C(6)$-assumption and thus contribute nonpositively to curvature. Any face that is not an $R$-face is a cone or a star and thus is incident at an apex. Consider an apex $a$. Each face incident at $a$ that comes from a boundary block has degree at least 4. Each face incident at $a$ that comes from an interior block has degree at least 5. By construction, $k:=d(a)\geqslant 8$. The number of faces incident at $a$ that come from boundary blocks is at most $\frac{k}{4}$, and the number of faces that come from interior blocks is at most $\frac{3k}{4}$. Thus the subdiagram incident at $a$ contributes at most $3-k+\frac{3k}{4}\frac{1}{2}(6-5)+\frac{k}{4}\frac{1}{2}(6-4)=3-\frac{3k}{8}\leqslant 0$ to the right-hand side of the curvature formula. We now sum over all apices (leaving out faces that have already been counted, which does not change the fact that the contribution to the right-hand side is nonpositive), to get:
$$0\geqslant \sum_{v\in\Sigma^{'(0)}}(3-d(v))+\frac{1}{2}\sum_{\Pi\in\Sigma^{'(2)}}(6-d(\Pi)).$$
This is a contradiction to the curvature formula, whence $N=\langle N\rangle^G\cap F$. 

\vspace{12pt}

To prove that $F$ is free and freely generated by the injective image of the set $\alpha$, let $w$ be a cyclically reduced freely nontrivial word in $\alpha\sqcup \alpha^{-1}$ with a diagram $D$ over $\langle S,\alpha\mid W,R\rangle$ for $w$ of minimal $(W,R)$-area and minimal number of edges (as above). Then $D$ is a simple disk diagram. We can again glue on a face $\Pi$ whose label is $w$ to obtain a simple spherical diagram. Replacing $\Pi$ by a wheel as above again gives a contradiction to the curvature formula.
\end{proof}

%% file: graphical_product2.tex
\section{SQ-universality of graphical $Gr_*(6)$-groups}\label{section:graphical}

In this section, we extend Theorem~\ref{thm:c6} to graphical small cancellation presentations over free products.

\begin{thm}\label{thm:gr6} 
 Let $\Gamma$ be a $Gr_*(6)$-labelled graph over a free product of infinite groups that has at least 16 pairwise non-isomorphic finite components with nontrivial fundamental groups. Then $G(\Gamma)$ is SQ-universal.
\end{thm}

We say two components are non-isomorphic if their completions (see Definition~\ref{defi:completion}) are non-isomorphic as labelled graphs. We say a component $\Gamma_0$ has non-trivial fundamental group if the set of labels of closed paths in $\Gamma_0$ is nontrivial in the free product.

We first recall definitions from graphical small cancellation theory presented in \cite{Gru} and then extend them to presentations over free products.

\subsection{Graphical small cancellation conditions}\label{subsection:definitions_graphical}
Recall Definition~\ref{defi:labelling} of a labelling of a graph $\Gamma$. A labelling is \emph{reduced} if the labels of reduced paths are freely reduced words. The \emph{group defined by} $\Gamma$ is given by the following presentation:
$$G(\Gamma):=\langle S\mid\text{labels of simple closed paths in }\Gamma\rangle.$$

\begin{defi}[Essential] Let $p$ and $\Gamma$ be labelled graphs, and let $\phi_1:p\to\Gamma$ and $\phi_2:p\to\Gamma$ be label-preserving graph homomorphisms. We say $\phi_1$ and $\phi_2$ are \emph{essentially equal} (or \emph{essentially coincide}) if there exists a label-preserving automorphism $\psi:\Gamma\to\Gamma$ such that $\phi_2=\psi\circ\phi_1$. Otherwise, $\phi_1$ and $\phi_2$ are \emph{essentially distinct}. 
\end{defi}

We will use the following notion of \emph{piece}. Note that a piece in the classical sense (as used in Section~\ref{section:classical}) is an essential piece in the graphical sense.

\begin{defi}[Piece]\label{defi:piece} Let $\Gamma$ be a labelled graph. 
\begin{itemize}
 \item A \emph{piece} is a labelled path $p$ (considered as labelled graph) for which there exist two distinct label-preserving graph homomorphisms $p\to\Gamma$. 
 \item An \emph{essential piece} is a labelled path $p$ for which there exist two essentially distinct label-preserving graph homomorphisms $p\to\Gamma$.
\end{itemize}
\end{defi}

We give the graphical small cancellation conditions stated in \cite{Gru}. A path is non-trivial if it is not 0-homotopic. 
\begin{defi} Let $n\in\N$ and $\lambda>0$. Let $\Gamma$ be a labelled graph. We say $\Gamma$ satisfies
\begin{itemize}
 \item the \emph{graphical $C(n)$-condition} (respectively \emph{graphical $Gr(n)$-condition}) if the labelling is reduced and no nontrivial closed path is concatenation of fewer than $n$ (essential) pieces,
 \item the \emph{graphical $C'(\lambda)$-condition} (respectively \emph{graphical $Gr'(\lambda)$-condition}) if the labelling is reduced and every (essential) piece $p$ that is a subpath of a nontrivial simple closed path $\gamma$ satisfies $|p|<\lambda|\gamma|$.  
\end{itemize}
\end{defi}

A \emph{diagram over $\Gamma$} is a diagram over the presentation $\langle S\mid \text{labels of simple closed}$ $\text{paths in }\Gamma\rangle$. Let $p$ be path in a diagram over $\Gamma$ that lies in the intersection of faces $\Pi$ and $\Pi'$. There exist \emph{lifts} $p\to\partial\Pi\to\Gamma$ and $p\to\partial\Pi'\to\Gamma$, which are essentially unique if $\Gamma$ is $Gr(2)$-labelled. We say $p$ \emph{essentially originates from $\Gamma$} if these lifts essentially coincide. Note that if an interior arc does not essentially originate from $\Gamma$, then it is an essential piece. The following is a version of van Kampen's Lemma stated in \cite[Lemma 2.13]{Gru}:

\begin{lem}\label{lem:graphical_basic} Let $\Gamma$ be a $Gr(6)$-labelled graph, and let $w$ be a word in $M(S)$. Then $w$ represents the identity in $G(\Gamma)$ if and only if there exists a diagram $D$ over $\Gamma$ such that no interior edge of $D$ essentially originates from $\Gamma$ and such that $D$ has boundary word $w$.
\end{lem}

\subsection{Graphical small cancellation over free products}\label{subsection:definitions_products}

Given groups $G_i,i\in I$ with generating sets $S_i,i\in I$, let $\Gamma$ be a graph labelled by the set $\sqcup_{i\in I}S_i$. In this situation, we define $G(\Gamma)$ to be the quotient of $\freeproduct_{i\in I}G_i$ by the normal subgroup generated by all labels of closed paths in $\Gamma$. 

Graphical small cancellation groups over free products were first studied in \cite{Ste}. The definitions we present here are slightly more general and provide a convenient way to skip notions such as ``reduced forms'' and ``semi-reduced forms'' used in standard definitions of small cancellation conditions over free products. These notions are used due to the fact that when writing an element of $G_i$ in the alphabet $S_i$, one has to \emph{choose} a word. While in a free group, every group element admits a single canonical representative, in an arbitrary group this is not the case. When constructing a graphical presentation over a free product, these choices give rise to distinct labelled graphs, as illustrated in Figure~\ref{figure:choices}. This is discussed in detail in \cite{Ste} where the terminology ``AO-move'' and a notion of equivalence of graphs are used. Our definition gives a single canonical object that contains all the possible choices.

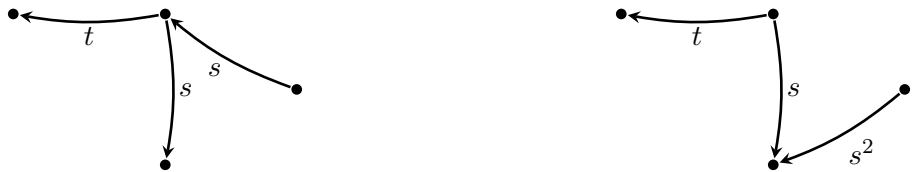
\begin{figure}\label{figure:choices}
 \begin{tikzpicture}[>=stealth,shorten <=2.5pt, shorten >=2.5pt,line width=1pt]
\coordinate (A) at (-9,1);
\coordinate (B) at (-7,1);
\coordinate (C) at (-7,-1);
\coordinate (BC) at (-5.27,0);
\fill (A) circle (2pt);
\fill (B) circle (2pt);
\fill (C) circle (2pt);
\fill (BC) circle (2pt);

\draw[->] (B) to [bend left=10] node [below=-.08cm] {\small $t$} (A);

\draw[->] (BC) to [bend left=10] node [below left=-.08cm] {\small $s$} (B);


\draw[->] (B) to [bend left=10] node [right=-.08cm] {\small $s$} (C);

\coordinate (A) at (-1,1);
\coordinate (B) at (1,1);
\coordinate (C) at (1,-1);
\coordinate (BC) at (2.73,0);
\fill (A) circle (2pt);
\fill (B) circle (2pt);
\fill (C) circle (2pt);
\fill (BC) circle (2pt);

\draw[->] (B) to [bend left=10] node [below=-.08cm] {\small $t$} (A);


\draw[->] (BC) to [bend left=10] node [below right=-.08cm] {\small $s^2$} (C);

\draw[->] (B) to [bend left=10] node [right=-.08cm] {\small $s$} (C);

\end{tikzpicture}
\caption{The group elements represented by the labels of paths between any two vertices are the same in both graphs.}
\end{figure}

\begin{defi}\label{defi:completion} Let $\Gamma$ be a graph labelled over $\sqcup_{i\in I}S_i$, where $S_i$ are generating sets of groups $G_i$. Denote by $\overline \Gamma$ the \emph{completion} of $\Gamma$ obtained as follows: Onto every edge labelled by $s_i\in S_i$, \emph{attach} a copy of $\Cay(G_i,S_i)$ along an edge of $\Cay(G_i,S_i)$ labelled by $s_i$. $\overline \Gamma$ is defined as the quotient of the resulting graph by the following equivalence relation: For edges $e$ and $e'$, we define $e\sim e'$ if $e$ and $e'$ have the same label and if there exists a path from $\iota e$ to $\iota e'$ whose label is trivial in $\freeproduct_{i\in I}G_i$. 
\end{defi}

As a motivation for the definition, one may consider a graph labelled over a free group as a graph with an immersion into a $K(G,1)$-space of a free group, i.e. a wedge of circles. Similarly, we consider a graph labelled over a free product as a graph with an immersion into a $K(G,1)$-space of a free product, i.e. a wedge of presentation complexes of the $G_i$.

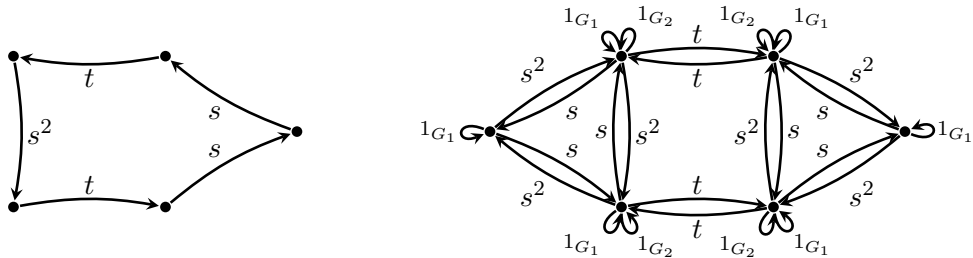
\begin{figure}
\begin{tikzpicture}[>=stealth,shorten <=2.5pt, shorten >=2.5pt,line width=1pt]
\coordinate (A) at (-9,1);
\coordinate (B) at (-7,1);
\coordinate (C) at (-7,-1);
\coordinate (D) at (-9,-1);
\coordinate (AD) at (-10.73,0);
\coordinate (BC) at (-5.27,0);
\fill (A) circle (2pt);
\fill (B) circle (2pt);
\fill (C) circle (2pt);
\fill (D) circle (2pt);
\fill (BC) circle (2pt);

\draw[->] (B) to [bend left=10] node [below=-.08cm] {\small $t$} (A);

\draw[->] (BC) to [bend left=10] node [below left=-.08cm] {\small $s$} (B);

\draw[->] (C) to [bend left=10] node [above left=-.08cm] {\small $s$} (BC);

\draw[->] (A) to [bend left=10] node [right=-.08cm] {\small $s^2$} (D);

\draw[->] (D) to [bend left=10] node [above=-.08cm] {\small $t$} (C);

\coordinate (A) at (-1,1);
\coordinate (B) at (1,1);
\coordinate (C) at (1,-1);
\coordinate (D) at (-1,-1);
\coordinate (AD) at (-2.73,0);
\coordinate (BC) at (2.73,0);
\fill (A) circle (2pt);
\fill (B) circle (2pt);
\fill (C) circle (2pt);
\fill (D) circle (2pt);
\fill (BC) circle (2pt);
\fill (AD) circle (2pt);

\draw[->] (A) to [bend left=10] node [above=-.08cm] {\small $t$} (B);
\draw[->] (B) to [bend left=10] node [below=-.08cm] {\small $t$} (A);

\draw[->] (B) to [bend left=10] node [above right=-.08cm] {\small $s^2$} (BC);
\draw[->] (BC) to [bend left=10] node [below left=-.08cm] {\small $s$} (B);

\draw[->] (BC) to [bend left=10] node [below right=-.08cm] {\small $s^2$} (C);
\draw[->] (C) to [bend left=10] node [above left=-.08cm] {\small $s$} (BC);

\draw[->] (C) to [bend left=10] node [left=-.08cm] {\small $s^2$} (B);
\draw[->] (B) to [bend left=10] node [right=-.08cm] {\small $s$} (C);

\draw[->] (D) to [bend left=10] node [below left=-.08cm] {\small $s^2$} (AD);
\draw[->] (AD) to [bend left=10] node [above right=-.08cm] {\small $s$} (D);

\draw[->] (AD) to [bend left=10] node [above left=-.08cm] {\small $s^2$} (A);
\draw[->] (A) to [bend left=10] node [below right=-.08cm] {\small $s$} (AD);

\draw[->] (A) to [bend left=10] node [right=-.08cm] {\small $s^2$} (D);
\draw[->] (D) to [bend left=10] node [left=-.08cm] {\small $s$} (A);

\draw[->] (C) to [bend left=10] node [below=-.08cm] {\small $t$} (D);
\draw[->] (D) to [bend left=10] node [above=-.08cm] {\small $t$} (C);

\draw[->] (A) to [out=100,in=145,loop] node [above left=-.08cm] {\tiny $1_{G_1}$} (A);
\draw[->] (A) to [out=45,in=90,loop] node [above right=-.08cm] {\tiny $1_{G_2}$} (A);

\draw[->] (B) to [out=90,in=135,loop] node [above left=-.08cm] {\tiny $1_{G_2}$} (B);
\draw[->] (B) to [out=35,in=80,loop] node [above right=-.08cm] {\tiny $1_{G_1}$} (B);

\draw[->] (C) to [out=280,in=325,loop] node [below right=-.08cm] {\tiny $1_{G_1}$} (C);
\draw[->] (C) to [out=225,in=270,loop] node [below left=-.08cm] {\tiny $1_{G_2}$} (C);

\draw[->] (D) to [out=270,in=315,loop] node [below right=-.08cm] {\tiny $1_{G_2}$} (D);
\draw[->] (D) to [out=215,in=260,loop] node [below left=-.08cm] {\tiny $1_{G_1}$} (D);

\draw[->] (AD) to [out=157.5,in=202.5,loop] node [left=-.08cm] {\tiny $1_{G_1}$} (AD);
\draw[->] (BC) to [out=337.5,in=22.5,loop] node [right=-.08cm] {\tiny $1_{G_1}$} (BC);

\end{tikzpicture}

\caption{An example of $\Gamma$ and $\overline \Gamma$:  $G_1=\Z/3\Z$ and $S_1=G_1=\{1_{G_1},s,s^2\}$, and $G_2=\Z/2\Z$ and $S_2=G_2=\{1_{G_2},t\}$. A presentation for $G(\Gamma)$ is given by $\langle s,t\mid s^3,t^2,(st)^2\rangle$.
}
\label{figure:example_free_product}
\end{figure}

See Figure~\ref{figure:example_free_product} for an example. Note that $\overline \Gamma$ is a reduced labelled graph by definition. We use the same notion of (essential) piece as above. A path $p$ in $\overline \Gamma$ is \emph{locally geodesic} if every subpath of $p$ that is contained in one of the attached $\Cay(G_i,S_i)$ is geodesic. More generally, a labelled path is locally geodesic if it is the isomorphic image of a locally geodesic path in $\overline\Gamma$.

\begin{defi} Let $n\in\N$ and $\lambda>0$. Let $\Gamma$ be labelled over $\sqcup_{i\in I}S_i$, where $S_i$ are generating sets of groups $G_i$. We say $\Gamma$ satisfies
\begin{itemize}
 \item the \emph{graphical $C_*(n)$-condition} (respectively \emph{graphical $Gr_*(n)$-condition}) if every attached $\Cay(G_i,S_i)$ in $\overline \Gamma$ is an embedded copy of $\Cay(G_i,S_i)$ and in $\overline \Gamma$ no path whose label is nontrivial in $\freeproduct_{i\in I}G_i$ is concatenation of fewer than $n$ (essential) pieces,
 \item the \emph{graphical $C_*'(\lambda)$-condition} (respectively \emph{graphical $Gr_*'(\lambda)$-condition}) if every attached $\Cay(G_i,S_i)$ in $\overline \Gamma$ is an embedded copy of $\Cay(G_i,S_i)$ and in $\overline \Gamma$ every (essential) piece $p$ that is locally geodesic and that is a subpath of a simple closed path $\gamma$ such that the label of $\gamma$ is nontrivial in $\freeproduct_{i\in I}G_i$ satisfies $|p|<\lambda|\gamma|$.  
\end{itemize}
\end{defi}

The first parts of the conditions correspond to a notion of ``reducedness'' of the labelling: The label of a nontrivial closed path is either trivial in $\freeproduct_{i\in I}G_i$ or not contained in one of the $G_i$. Classical free product small cancellation \cite[Chapter V]{LS} considers the case that $\Gamma$ is a disjoint union of cycle graphs and $S_i=G_i$. Graphical free product small cancellation with respect to arbitrary length functions was first considered in \cite{Ste}. If the groups $G_i$ are infinite cyclic and each $S_i$ is a set containing a single generator of $G_i$, we recover stronger versions of the graphical small cancellation conditions from Section~\ref{subsection:definitions_graphical} as discussed in Remark~\ref{remark:differences}. Note that for the graphical $C_*(n)$- and $Gr_*(n)$-conditions, the choice of generating set $S_i$ used to construct $\overline\Gamma$ is irrelevant, i.e. we may take $S_i=G_i$.

The analogy of Lemma \ref{lem:graphical_basic} follows from the proofs of \cite[Lemma 2.13]{Gru} and \cite[Theorem 1.11]{Ste}. A \emph{diagram over $\overline \Gamma$} is a diagram where every face $\Pi$ either bears the label of a simple closed path in $\overline\Gamma$ that is nontrivial in $\freeproduct_{i\in I}G_i$, or $\Pi$ bears the label of a simple closed path in some $\Cay(G_i,S_i)$ and has no interior edge.

\begin{lem}\label{lem:graphical_product_basic}
 Let $\Gamma$ be a $Gr_*(6)$-labelled graph over $S=\sqcup_{i\in I} S_i$, where $S_i$ are generating sets of groups $G_i$, and let $w$ be a word in $M(S)$. Then $w$ represents the identity in $G(\Gamma)$ if and only if there exists a diagram $D$ over $\overline\Gamma$ such that no interior edge of $D$ essentially originates from $\Gamma$, such that every interior arc is locally geodesic, and such that $D$ has boundary word $w$.
\end{lem}

\begin{remark}[Graphical $Gr$ vs. $Gr_*$ conditions]\label{remark:differences}
Suppose $\Gamma$ is a graph labelled over the generating set $S$ of the free group $F(S)$ considered as the free product $*_{s\in S}F(s)$, and suppose $\Gamma$ does not have a vertex of degree 1. (Not having a vertex of degree 1 ensures that $\Gamma$ does not have unnecessary edges that restrict its automorphism group.) See Figure~\ref{figure:differences} for an illustration.

If $p$ and $q$ are essentially distinct reduced paths with the same label on $\Gamma$ and the terminal vertices of $p$ and $q$ are both incident at edges with label $s$, then $ps$ and $qs$, and also $ps^{-1}$ and $qs^{-1}$ are essential pieces in $\overline\Gamma$, and the same holds for the initial vertices. Conversely, essentially distinct paths $p'$ and $q'$ in $\overline \Gamma$ with label $s_1^lws_2^k$, $k,l\in \Z\setminus\{0\}$, where the initial letter of $w$ is not $s_1^{\pm1}$ and the terminal letter is not $s_2^{\pm1}$, contain subpaths $p$ and $q$ with label $w$ that also exist in $\Gamma$ and are essentially distinct in $\Gamma$. Thus, a sufficient for a condition for a $Gr(n)$-labelled graph to satisfy the $Gr_*(n)$-condition is the following: No nontrivial closed path can be written as $p_1q_1p_2q_2\dots p_{n-1}q_{n-1}$, where each $p_i$ is an essential piece 
in $\Gamma$ and each $q_i$ is labelled by a product of at most two powers of generators.

If $\Gamma$ satisfies the $Gr_*(n)$-condition over a free product of free groups (with respect to free generating sets), then $\overline \Gamma$ satisfies the $Gr(n)$-condition since there are no nontrival closed paths in any attached $\Cay(G_i,S_i)$.
\end{remark}

\begin{figure}\label{figure:differences}
\begin{tikzpicture}[>=stealth,
shorten <=2.5pt, shorten >=2.5pt,
line width=1pt,x=.5cm,y=.5cm, scale=.8]
\coordinate (A) at (-12,0);
\coordinate (B) at (-10,0);
\coordinate (C) at (-8.59,1.41);
\coordinate (D) at (-8.59,3.41);
\coordinate (E) at (-10.59,3.41);
\coordinate (F) at (-12,2);

\fill (A) circle (2pt);
\fill (B) circle (2pt);
\fill (C) circle (2pt);
\fill (D) circle (2pt);
\fill (E) circle (2pt);
\fill (F) circle (2pt);

\draw[->-] (A) to (B);
\draw[->-] (B) to (C);
\draw[->>-] (C) to (D);
\draw[-<-] (D) to (E);
\draw[-<<-] (E) to (F);
\draw[-<<-] (F) to (A);

\fill (0,0) circle (2pt);
\fill (2,0) circle (2pt);
\fill (3.41,1.41) circle (2pt);
\fill (3.41,3.41) circle (2pt);
\fill (1.41,3.41) circle (2pt);
\fill (0,2) circle (2pt);

\draw[->-] (0,0) to (2,0);
\draw[->-] (2,0) to (3.41,1.41);
\draw[->>-] (3.41,1.41) to (3.41,3.41);
\draw[-<-] (3.41,3.41) to (1.41,3.41);
\draw[-<<-] (1.41,3.41) to (0,2);
\draw[-<<-] (0,2) to (0,0);

\fill (-2,0) circle (2pt); \draw[->-] (-2,0) to (0,0); \draw[->-] (-4,0) to (-2,0); \node at (-4.5,0) {$\dots$};
\fill (0,-2) circle (2pt); \draw[->>-] (0,-2) to (0,0); \draw[->>-] (0,-4) to (0,-2); \node at (0,-4.3) {$\vdots$};

\fill (5.41,1.41) circle (2pt); \draw[->-] (3.41,1.41) to (5.41,1.41);\draw[->-] (5.41,1.41) to (7.41,1.41); \node at (7.91,1.41) {$\dots$};
\fill (5.41,3.41) circle (2pt); \draw[->-] (3.41,3.41) to (5.41,3.41);\draw[->-] (5.41,3.41) to (7.41,3.41); \node at (7.91,3.41) {$\dots$};
\fill (-.59,3.41) circle (2pt); \draw[->-] (-.59,3.41) to (1.41,3.41);\draw[->-] (-2.59,3.41) to (-.59,3.41); \node at (-3.09,3.41) {$\dots$};

\fill (1.41,5.41) circle (2pt); \draw[->>-] (1.41,3.41) to (1.41,5.41); \draw[->>-] (1.41,5.41) to (1.41,7.41); \node at (1.41,7.91) {$\vdots$}; 
\fill (3.41,5.41) circle (2pt); \draw[->>-] (3.41,3.41) to (3.41,5.41); \draw[->>-] (3.41,5.41) to (3.41,7.41); \node at (3.41,7.91) {$\vdots$}; 
\fill (3.41,-.59) circle (2pt); \draw[->>-] (3.41,-.59) to (3.41,1.41); \draw[->>-] (3.41,-2.59) to (3.41,-.59); \node at (3.41,-2.89) {$\vdots$}; 
\end{tikzpicture}
\caption[]{An example of $\Gamma$ (left) and $\overline \Gamma$ (right) over the free product $G_1*G_2$, where $G_1$ is the infinite cyclic group generated by $S_1=\{a\}$ and $G_2$ is the infinite cyclic group generated by $S_2=\{b\}$. In the picture, $a$ is represented by \tikz[>=stealth,shorten <=2.5pt, shorten >=2.5pt,line width=1pt]{\fill (.5,-.1) circle (0pt); \draw[->-] (0,0) to (1,0);} and $b$ is represented by \tikz[>=stealth,shorten <=2.5pt, shorten >=2.5pt,line width=1pt]{\fill (.5,-.1) circle (0pt); \draw[->>-] (0,0) to (1,0);}. Note that $\Gamma$ satisfies the graphical $Gr(6)$-condition, since every essential piece has length at most 1. On the other hand, $\Gamma$ does not satisfy the graphical $Gr_*(6)$-condition: For example, in $\overline\Gamma$ there exist paths labelled $a^2b$ and $a^{-1}b^{-2}$ that are essential pieces and whose concatenation is a closed path in $\overline\Gamma$ with nontrivial label in $G_1*G_2$. 
}
\end{figure}

\subsection{Proof of Theorem~\ref{thm:gr6}}

An \emph{interior vertex} in an attached Cayley graph is a vertex that is not contained in any other attached Cayley graph. We say a component of $\overline\Gamma$ is \emph{finite} if it has only finitely many vertices that are not interior vertices of attached Cayley graphs, and we say it has \emph{non-trivial fundamental group} if the set of labels of closed paths is nontrivial in the free product.

\begin{lem}\label{lem:graphial_support} Let $\overline\Gamma$ be a $Gr_*(6)$-labelled graph over a free product of infinite groups. Let $x$ be a vertex in a finite component $\overline\Gamma_0$ of $\overline \Gamma$ with nontrivial fundamental group. Then there exists a vertex $y$ in $\overline\Gamma_0$ such that:
\begin{itemize}
 \item No path from $x$ to $y$ is concatenation of at most two essential pieces, and
 \item $y$ lies in the interior of an attached Cayley graph.
\end{itemize}
\end{lem}

\begin{proof}
 Since $\overline\Gamma_0$ is finite and every attached $\Cay(G_i,S_i)$ is infinite, the group of label-preserving automorphisms of $\overline\Gamma_0$ cannot operate transitively on any attached $\Cay(G_i,S_i)$. Therefore, every edge of $\overline\Gamma_0$ is an essential piece.
 
 Let $x$ be a vertex, and let $\overline\Gamma_0'$ the subgraph of $\overline\Gamma_0$ that is the union of all paths starting at $x$ that are concatenations of at most two essential pieces. Then $\overline\Gamma_0'$ has trivial fundamental group by the $Gr_*(6)$-assumption and, therefore, is a proper subgraph. By construction, $\overline\Gamma_0'$ is a union of attached Cayley graphs, and the subgraph $\overline\Gamma_0''$ of $\overline\Gamma_0$ whose edges are the edges not contained in $\overline\Gamma_0'$ is a union of attached Cayley graphs. Choosing $y$ in the interior of an attached Cayley graph in $\overline\Gamma_0''$ yields the claim.
\end{proof}

\begin{proof}[Proof of Theorem~\ref{thm:gr6}.] There exist at 16 least pairwise non-isomorphic finite components of $\overline\Gamma$ with non-trivial fundamental groups. By Lemma~\ref{lem:graphial_support}, we may choose vertices $x_i,y_i$, $i\in\{1,2,\dots,16\}$ in each component with $d_p(x_i,y_i)\geqslant3$ such that $y_i$ and $x_{i+1}$ never lie in attached Cayley graphs corresponding to the same $G_i$. Thus $y_i$ and $x_{i+1}$ have disjoint support. 

We make definitions as those leading up to Proposition~\ref{prop:main_classical}. When defining $W_i$, we take the $p_k$ to be \emph{all} paths in $\overline \Gamma$ from $x_k$ to $y_k$. We take $R$ to be the set of \emph{all} words read on closed paths in $\overline \Gamma$. We carry out the proof of Proposition~\ref{prop:main_classical} with only the following additional observations: 

When considering $D$, we can assume that each $R$-face has a boundary word that is nontrivial in $*_{i\in I}G_i$. If this is not the case for a face $\Pi$ then, by Lemma~\ref{lem:reduction}, we can replace $\Pi$ by a diagram made up of faces $\Pi_1,\dots \Pi_l$ such that each $\partial\Pi_i$ lifts to a closed path contained in one of the attached $\Cay(G_i,S_i)$. Observe that if such a face  $\Pi_i$ intersects another face $\tilde\Pi$ in an edge, by our definitions, we can merge $\Pi_i$ into $\tilde\Pi$. Thus, we can merge all faces $\Pi_i$ into other faces, and, hence, the existence of $\Pi$ contradicts minimality. Therefore every $R$-face lifts to a nontrivial closed path and, hence, has a boundary path made up of no fewer than $6$ essential pieces.

When considering an arc $a$ in the intersection of two $R$-faces, we can assume that it does not essentially originate from $\Gamma$, for else, we could remove the arc $a$ to obtain a single $R$-face, contradicting minimality. Therefore $a$ is an essential piece. The rest of claims 1 and 2 follows with the same proofs, replacing the word ``piece'' by ``essential piece''.
\end{proof}

%% file: undistorted_subgroups3.tex
\section{Cyclic subgroups of graphical $Gr'(\frac{1}{6})$-groups are undistorted}

In this section, we show that every cyclic subgroup of a graphical $Gr'(\frac{1}{6})$-group is undistorted. We use this result to construct for every $p$ uncountably many classical $C(p)$-groups that do not admit any graphical $Gr'(\frac{1}{6})$-presentation.

\begin{defi}
 Let $G$ be a group generated by a finite set $S$, and let $H$ be a subgroup. We say $H$ is \emph{undistorted} in $G$ if $H$ is finitely generated and the inclusion $H\to G$ is a quasi-isometric embedding with respect to the corresponding word-metrics. We say $H$ is \emph{quasi-convex} in $G$ with respect to $S$ if there exists $C>0$ such that every geodesic in $\Cay(G,S)$ connecting two elements of $H$ is contained in the $C$-neighborhood of $H$.  
\end{defi}

\begin{thm}\label{thm:undistorted}
 Suppose the set of labels $S$ is finite. Let $\Gamma$ be a $C'(\frac{1}{6})$-labelled graph, or let $\Gamma$ be a $Gr'(\frac{1}{6})$-labelled graph whose components are finite. Then every cyclic subgroup of $G(\Gamma)$ is undistorted and conjugate to a cyclic subgroup that is quasi-convex with respect to $S$.
\end{thm}

\begin{thm}\label{thm:undistorted_product}
 Suppose $I$ is finite. Let $\Gamma$ be a $C_*'(\frac{1}{6})$-labelled graph over a free product $*_{i\in I}G_i$ with respect to finite generating sets $S_i$, or let $\Gamma$ be a $Gr_*'(\frac{1}{6})$-labelled graph with finite components over a free product $*_{i\in I}G_i$ with respect to finite generating sets $S_i$. Then: 
 \begin{itemize}
  \item Every cyclic subgroup of $G(\Gamma)$ is undistorted if and only if for every $i$, every cyclic subgroup of $G_i$ is undistorted.
  \item Every cyclic subgroup of $G(\Gamma)$ is conjugate to a cyclic subgroup with that is quasi-convex with respect to $\sqcup_{i\in I} S_i$ if and only if for every $i$, every cyclic subgroup of $G_i$ is conjugate to a quasi-convex cyclic subgroup with respect to $S_i$.
 \end{itemize}
 \end{thm}

 
 We postpone the proofs of Theorems~\ref{thm:undistorted} and \ref{thm:undistorted_product} and first give two consequences. Since every classical $C'(\frac{1}{6})$-presentation corresponds to a $Gr'(\frac{1}{6})$-labelled graph $\Gamma$ where every component is a finite cycle graph, Theorem~\ref{thm:undistorted} implies:

\begin{cor}
 Let $G$ be a group admitting a classical $C'(\frac{1}{6})$-presentation with finite generating set. Then every cyclic subgroup is undistorted.
\end{cor}

This corollary can also be deduced from the facts that every infinitely presented classical $C'(\frac{1}{6})$-group acts properly on a CAT(0) cube complex \cite{AO} and that every group that acts properly on a CAT(0) cube complex has no distorted cyclic subgroups \cite{Hag}.

We now construct classical $C(p)$-groups with distorted cyclic subgroups.

\begin{example}\label{example:distorted}
 Consider (the symmetrized closure of) the presentation $P=\langle a,b\mid r_n,n\in\N\rangle$, where
 $$r_n:=a b^{2np+1}a b^{2np+3}a\dots ab^{2np+2p-1}ab^{2^{n}}.$$
 This is a classical $C(p)$-presentation: Every piece is a subword of a word of the form $b^ka^{\pm 1}b^l$ with $k,l\in \Z$ and, hence, contains at most one copy of the letter $a$. Since every $r_n$ contains $p+1$ copies of the letter $a$, the classical $C(p)$-condition is satisfied.
 
Denote by $G$ the group defined by $P$. The cyclic subgroup of $G$ generated by $b$ is distorted: Since cycle graphs labelled by $r_n$ embed into $\Cay(G,\{a,b\})$ by \cite[Lemma 4.1]{Gru}, it is infinite. The group element represented by $b^{2^n}$ can be represented by a word of length $O(n)=o(2^n)$.
 
 If $p>8$, then $P$ satisfies the $Gr'_*(\frac{1}{6})$-condition with respect to $\langle a\rangle *\langle b \rangle$ with \emph{infinite} generating sets $\langle a\rangle $ and $\langle b\rangle$. (Here we consider $P$ as a disjoint union of labelled cycle graphs.) Thus, the restriction in Theorem~\ref{thm:undistorted_product} that the generating sets are finite is necessary. 

 If $p>8$, $P$ satisfies the \emph{classical} $C'_*(\frac{1}{6})$-condition over the free product $\langle a\rangle *\langle b\rangle$ as in \cite[Chapter V]{LS}. By the aforementioned results of Haglund \cite{Hag}, $G$ cannot act properly on a CAT(0) cube complex. A recent result of Martin and Steenbock shows that every finitely presented classical $C_*'(\frac{1}{6})$-group acts properly cocompactly on a CAT(0) cube complex if every generating free factor does. Our example shows that this does not extend in any way to infinite presentations. This contrasts the situation for classical $C'(\frac{1}{6})$-groups: By \cite{Wise}, every finitely presented classical $C'(\frac{1}{6})$-group acts properly cocompactly on CAT(0) cube complex, and by \cite{AO} every infinitely presented classical $C'(\frac{1}{6})$-group acts properly on a CAT(0) cube complex. 
\end{example}

 Combining a construction of Bowditch \cite{Bow} with Example~\ref{example:distorted}, we use Theorem~\ref{thm:undistorted} to show:

\begin{thm}\label{thm:uncountable}
 Given $p\in\N$, there exist uncountably many pairwise non-quasi-isometric finitely generated groups $(G_i)_{i\in I}$ such that:
 \begin{itemize}
  \item Every $G_i$ admits a classical $C(p)$-presentation with a finite generating set.
  \item No $G_i$ is isomorphic to any group defined by a $C'(\frac{1}{6})$-labelled graph with a finite set of labels.
  \item No $G_i$ is isomorphic to any group defined by $Gr'(\frac{1}{6})$-labelled graph whose components are finite with a finite set of labels.
 \end{itemize}
\end{thm}
 
\begin{proof}[Proof of Theorem \ref{thm:uncountable}.]
Without loss of generality, let $p\geqslant 7$. Let $G$ be the classical $C(p)$-group with distorted cyclic subgroups defined in Example~\ref{example:distorted}. Note that $G$ is one-ended by Stallings' theorem, since it is torsion-free but not free. Let $\{H_i|i\in I\}$ be an uncountable set of pairwise non-quasi-isometric 1-ended groups, each given by a classical $C'(\frac{1}{p-1})$-presentation as in \cite{Bow}. Consider the groups $G_i:=G*H_i$. By \cite[Theorem 0.4]{PW}, two free products of 1-ended groups $A*B$ and $A'*B'$ are quasi-isometric if and only if $\{[A],[B]\}=\{[A'],[B']\}$, where $[\cdot]$ denotes the quasi-isometry class of a group. This shows that the $G_i$ are pairwise non-quasi-isometric.

Each $G_i$ contains distorted cyclic subgroups and admits a classical $C(p)$-presentation with a finite generating set. Theorem~\ref{thm:undistorted} yields the remaining claims.
\end{proof}

The rest of this subsection is devoted to the proof of Theorem~\ref{thm:undistorted}. The proof of Theorem~\ref{thm:undistorted_product} will follow in a very brief remark at the end. The two cases of Theorem~\ref{thm:undistorted} are:
\begin{itemize}
 \item the graphical $Gr'(\frac{1}{6})$-case: $\Gamma=\sqcup_{i\in I}\Gamma_i$, where the $\Gamma_i$ are the connected components of $\Gamma$, and all $\Gamma_i$ are finite,
 \item the graphical $C'(\frac{1}{6})$-case.
\end{itemize}
We discuss in detail the first case. The second case is obtained as a simplification by skipping all considerations involving non-trivial automorphisms of $\Gamma$. 

Let $\Gamma$ be a $Gr'(\frac{1}{6})$-labelled graph with finite components labelled over a finite set $S$. Let $h\in G(\Gamma)$ be of infinite order. Let $g$ be an element of minimal word-length with respect to $S$ in the set $\{g'\in G(\Gamma)\mid \exists k\in \Z: (g')^k\text{ is conjugate to } h\}$. We prove that a conjugate of $\langle g \rangle$ is undistorted and quasi-convex in $\Cay(G(\Gamma),S)$. 

Let $w$ be a word of minimal length representing $g$. Note that $w$ is cyclically reduced and not a proper power. Denote by $\overline w$ the bi-infinite ray in $\Cay(G(\Gamma),S)$ that is the union of all paths starting at $1\in\Cay(G(\Gamma),S)$ that are labelled by initial subwords of all words of the form $w^k,k\in \Z$. A \emph{subword of} $\overline w$ is the label of a finite reduced path contained in $\overline w$. 
We consider three cases:

\vspace{12pt}

\noindent {\bf Case 1.} There does not exist $C_0<\infty$ such that every path $p$ in $\Gamma$ that is labelled by a subword of $\overline w$ has length at most $C_0$.

\vspace{12pt}

A \emph{period} of $\Gamma$ is a path $p$ in $\Gamma$ with $\iota p\neq \tau p$ such that there exists a label-preserving automorphism $\phi$ of $\Gamma$ with $\phi(\iota p)=\phi(\tau p)$. Note that if a period $p$ is contained in a finite component of $\Gamma$, then the label of $p$ defines a finite order element of $G(\Gamma)$.

\begin{proof}[Proof of Theorem~\ref{thm:undistorted} in case 1.]
 Suppose a path $p$ labelled by a subword of $\overline w$ has length at least $2|w|$. Since $w$ defines an infinite order element of $G(\Gamma)$, a path labelled by a cyclic conjugate of $w$ cannot be closed and it cannot be a period. Therefore, $p$ is concatenation of at most two essential pieces. Any path in $\Gamma$ that is concatenation of at most two essential pieces is a convex geodesic, since any other path with the same endpoints and at most the same length would give rise to a simple closed path violating the $Gr'(\frac{1}{6})$-condition. Therefore, $p$ is a convex geodesic in $\Gamma$, and its image in $\Cay(G(\Gamma),S)$ is a convex geodesic by \cite[Lemma 4.21]{GS}.
\end{proof}

For each $n\in\N$, let $g_n$ be a shortest word representing $g^n$, and let $B_n$ be a diagram over $\Gamma$ with boundary word $w^ng_n^{-1}$, such that each face of $B_n$ bears the label of a simple cycle in $\Gamma$ and no interior edge of $B_n$ essentially originates from $\Gamma$. Such a diagram exists by Lemma~\ref{lem:graphical_basic}. We abuse notation to decompose $\partial B_n$ into subpaths $w^n$ and $g_n$. 

\vspace{12pt}

\noindent {\bf Case 2a.} There exists $C_0$ such that every path $p$ in $\Gamma$ that is labelled by a subword of $\overline w$ has length at most $C_0$, and for every simple closed path $\gamma$ in $\Gamma$, any subpath $p$ that is labelled by a subword of $\overline w$ satisfies $|p|\leqslant \frac{|\gamma|}{2}$.

\vspace{12pt}

We use the following fact shown in the proof of \cite[Theorem 35]{Str}. Let $D$ be a diagram.  Given a boundary face $\Pi$, the \emph{exterior part} of $\Pi$ is the union of all edges in $\Pi\cap \partial D$. The \emph{exterior degree} of $\Pi$, denoted $e(\Pi)$, is the number of arcs in $\Pi\cap \partial D$. The \emph{interior degree} of $\Pi$, denoted $i(\Pi)$, is the number of arcs in $\partial \Pi$ not contained in $\partial D$ counted with multiplicity. A diagram is a \emph{$(3,7)$-diagram} if every interior face $\Pi$ satisfies $i(\Pi)\geqslant 7$.

\begin{lem}[\cite{Str}]\label{lem:strebel}
 Let $D$ be a simple disk diagram such that $\partial D=\delta_1\delta_2$, where $\delta_1$ and $\delta_2$ are immersed paths such that:
 \begin{itemize}
  \item $D$ is a $(3,7)$-diagram.
  \item Any face $\Pi$ with $e(\Pi)=1$ whose exterior part is contained in either $\delta_1$ or in $\delta_2$ satisfies $i(\Pi)\geqslant 4$.
 \end{itemize}
Then $D$ is either a single face, or it has shape $I_1$ depicted in Figure~\ref{figure:type_I}.
\end{lem}

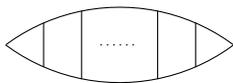
\begin{figure}\label{figure:type_I}
\begin{tikzpicture}[line cap=round,line join=round,x=1.0cm,y=1.0cm,scale=.5]
\draw [shift={(2,-3)}] plot[domain=0.93:2.21,variable=\t]({1*5*cos(\t r)+0*5*sin(\t r)},{0*5*cos(\t r)+1*5*sin(\t r)});
\draw [shift={(2,5)}] plot[domain=4.07:5.36,variable=\t]({1*5*cos(\t r)+0*5*sin(\t r)},{0*5*cos(\t r)+1*5*sin(\t r)});
\draw (0,1.58)-- (0,0.42);
\draw (4,1.58)-- (4,0.42);
\draw (1,1.9)-- (1,0.1);
\draw (3,1.9)-- (3,0.1);
\draw [dotted] (1.5,1)-- (2.5,1);
\end{tikzpicture}
\caption{A diagram $D$ of \emph{shape $I_1$}. All faces except the two extremal ones are optional, i.e. $D$ may have as few as 2 faces.}
\end{figure}

\begin{proof}[Proof of Theorem~\ref{thm:undistorted} in case 2a.]
In case 2a, $B_n$ has the following properties:
\begin{itemize}
 \item For any two faces $\Pi$ and $\Pi'$ of $B_n$, we have $|\partial \Pi\cap\partial \Pi'|<\frac{|\partial\Pi|}{6}$.
 \item $\partial B_n$ decomposes into two immersed paths $\delta_1:=w^n$ and $\delta_2:=g_n$, and any face $\Pi$ such that $\Pi\cap\partial B_n$ is an arc in $\delta_i$ ($i\in \{1,2\}$) satisfies $|\Pi\cap \delta_i|\leqslant \frac{|\partial \Pi|}{2}$.
\end{itemize}

Thus every disk component of $B_n$ satisfies the assumptions of Lemma~\ref{lem:strebel} and, hence, has the shape depicted in Figure~\ref{figure:type_I}. Thus, for every face $\Pi$ we have $|\Pi\cap g_n|>|\partial\Pi|-\frac{|\partial\Pi|}{2}-2\frac{|\partial\Pi|}{6}=\frac{|\partial\Pi|}{6}$. Therefore, 
\begin{equation*}\label{eqn:qi} 3|g_n|>|w^n|=n|w|,\end{equation*}
and in this case $\langle g\rangle$ is undistorted. 
Moreover, every $\Pi$ satisfies $\frac{|\partial \Pi|}{6}=|\partial \Pi|-\frac{|\partial \Pi|}{2}-2\frac{|\partial \Pi|}{6}<|\Pi\cap w_n| \leqslant C_0$. Therefore, in $B_n$, $g_n$ and $w^n$ are within Hausdorff-distance $C:=6C_0+|w|$ of each other. Since the 1-skeleton of $B_n$ maps to $\Cay(G(\Gamma),S)$, the images of $g_n$ and $w^n$ are also $C$-close in $\Cay(G(\Gamma),S)$.
\end{proof}


\noindent {\bf Case 2b.} There exists $C_0$ such that every path $p$ in $\Gamma$ that is labelled by a subword of $\overline w$ has length at most $C_0$, and there exists a simple closed path $\gamma$ in $\Gamma$ having a subpath $p$ that is labelled by a subword of $\overline w$ and satisfies $|p|> \frac{|\gamma|}{2}$.

\vspace{12pt}

Given a path $\gamma$, we write $\gamma\cap\overline w$ for a maximal subpath of $\gamma$ labelled by a subword of $\overline w$. Given a path $p$ and a natural number $k$, the subpath of $p$ \emph{exceeding} $k$ is the terminal subpath of $p$ that starts at the terminal vertex of the initial subpath of length $k$ of $p$.

\begin{lem}\label{lem:inequalities}
 Let $\gamma$ be a simple closed path in $\Gamma$ with $\frac{|\gamma|}{2}<|\gamma\cap\overline w|$. Then $|w|<|\gamma\cap\overline w|<|w|+\frac{|\gamma|}{6}$, $2|w|\leqslant|\gamma|<3|w|$, and $|\gamma\cap\overline w|<\frac{2|\gamma|}{3}$.
\end{lem}

\begin{proof}
 For the second part of the first inequality, note that the subpath of $\gamma\cap\overline w$ exceeding $|w|$ is an essential piece since any path labelled by a cyclic conjugate of $w$ is not a period. By the $Gr'(\frac{1}{6})$-assumption, this implies that $|\gamma\cap\overline w|<|w|+\frac{|\gamma|}{6}$. 
 
 The first part of the second inequality follows since no word representing a conjugate of $g$ can be shorter than $|w|$, whence $2|w|\leqslant |\gamma|$. The second part follows from the second part of the first inequality and the assumption that $\frac{|\gamma|}{2}<|\gamma\cap\overline w|$. The first part of the first inequality follows from the assumption that $\frac{|\gamma|}{2}<|\gamma\cap\overline w|$ and the first part of the second inequality. The third inequality follows immediately from the first two.
 \end{proof}


 \begin{cor}\label{cor:degree3}
 Let $\Pi$ be a face of $B_n$ such that $e(\Pi)=1$ and such that the exterior part of $\Pi$ is contained in $w^n$. Then $i(\Pi)\geqslant 3$.
\end{cor}

\begin{proof}
If a face $\Pi$ as above satisfies $|\Pi\cap w^n|\leqslant \frac{|\partial \Pi|}{2}$, then $i(\Pi)\geqslant 4$ by the small cancellation assumption. If $|\Pi\cap w^n|> \frac{|\partial \Pi|}{2}$, then $|\Pi\cap w^n|<\frac{2|\partial \Pi|}{3}$, whence $i(\Pi)\geqslant 3$, again by the small cancellation assumption.
\end{proof}
 
 

 

\begin{lem}\label{lem:sigma}
 In $\Gamma$ there exists a path $\sigma$ of length $|w|<|\sigma|<2|w|$ whose label is a subword of $\overline w$ such that for any path $p$ in $\Gamma$ whose label is a cyclic conjugate of $w$ there exists a unique subpath $p'$ of $\sigma$ (with the same orientation as $\sigma$) and a label-preserving automorphism $\phi$ of $\Gamma$ with $\phi(p)=p'$. 
\end{lem}
\begin{proof}
 Fix a path $\sigma'$ that is labelled by a cyclic conjugate of $w$ such that $\sigma'$ is contained in a simple closed path $\gamma$. Let $\delta$ be a path labelled by a cyclic conjugate of $w$. Then we may write $\ell(\sigma')\equiv uv$ and $\ell(\delta)\equiv vu$ ($\equiv$ denoting equality without cancellation). In particular $\max\{|v|,|u|\}\geqslant\frac{|w|}{2}$.  Since $|\gamma|<3|w|$ by Lemma~\ref{lem:inequalities}, a subpath of $\sigma'$ that is an essential piece cannot have length at least $\frac{|w|}{2}$. Therefore, there exists a label-preserving automorphism $\phi$ of $\Gamma$ such that $\phi(\delta)$ is contained in the maximal path $\sigma$ in $\Gamma$ that contains $\sigma'$ and that is labelled by a subword of $\overline w$.
 
All subpaths of $\sigma$ of length $|w|$ with the same orientation as $\sigma$ are labelled by cyclic conjugates of $w$. Since no word can be conjugate to its inverse in the free group, this shows that all subpaths of $\sigma$ of labelled by cyclic conjugates of $w$ have the same orientation as $\sigma$.
 
 Now assume $\delta$ is a subpath of $\sigma$ labelled by a cyclic conjugate of $w$, and suppose there exists an automorphism $\phi$ of $\Gamma$ such that $\phi(\delta)$ is a subpath of $\sigma$ with $\delta\neq \phi(\delta)$. 
 Consider the shortest subpath $\delta'$ of $\sigma$ containing both $\delta$ and $\phi(\delta)$, and denote by $k$ the finite order of $\phi$. Then the union $\delta'\cup\phi(\delta')\cup\dots\cup\phi^{k-1}(\delta')$ contains a closed path labelled by a power of (a cyclic conjugate of) $w$. Here we use that $w$ is not a proper power. This contradicts the fact that $g$ has infinite order. Therefore, the map $\delta\to\sigma$ is unique.
 
 The statement on the length of $\sigma$ follows from the uniqueness of the map taking initial subpath of $\sigma$ of length $|w|$ to $\sigma$.
\end{proof}

By cyclically conjugating $w$, we may assume that $w$ is the label of an initial subpath of $\sigma$. We from now on fix the notation for $\sigma$ (i.e. we choose one fixed $\sigma$).

Consider a diagram $B_n$. Let $\pi$ be an arc in the intersection of $w^n$ and a face $\Pi$. 
We call $\pi$ \emph{special} if there exists a 
map $\pi\to\sigma$ and for all maps $\pi\to\sigma$ the maps $\pi\to\partial\Pi\to\Gamma$ and $\pi\to\sigma\to\Gamma$ essentially coincide. 
We call a face special if $\Pi\cap w^n$ is a special arc. 

Let $D$ be a diagram, and let $\Pi$ and $\Pi'$ intersect $\partial D$. We call $\Pi$ and $\Pi'$ \emph{consecutive} if $\Pi\cap\Pi'$ contains an edge incident at $\partial D$.

\begin{lem}\label{lem:consecutive} Let $\Pi$ and $\Pi'$ be consecutive special faces of $B_n$. Let $a$ be an arc in $\Pi\cap\Pi'$ incident at $w^n$, and let $\pi=\Pi\cap w^n$ and $\pi'=\Pi'\cap w^n$. Suppose $|\pi|\geqslant |w|$. Then $|\pi|+|a|<|w|+\frac{|\partial\Pi|}{6}$.
\end{lem}
\begin{proof} 
We consider $\pi$ and $\pi'$ as paths with the orientation coming from $w^n$. By symmetry, we may assume that $\pi\pi'$ is a subpath of $w^n$. Choose the unique lift $\partial \Pi\to \Gamma$ that takes $\pi$ to $\sigma$. Let $\pi_0'$ be the maximal initial subpath of $\pi'$ of length at most $|w|$.
Then $\pi_0'$ has a map to $\sigma$ that takes $\iota \pi'$ to the vertex $v'$ in $\sigma$ that is at distance $|w|$ from the image $v$ of $\tau \pi$. (This means $v'$ sits at distance $|w|$ to the left of $v$, if we read the label of $\sigma$ from left to right.) If $\pi_0'=\pi'$, then, since $\pi'$ is special, we may extend this lift to a lift $\partial \Pi'\to\Gamma$. If $\pi'\neq \pi_0'$, then $|\pi_0'|=|w|$ and therefore, the map $\pi_0'\to\sigma$ is the essentially unique map $\pi_0'\to\Gamma$, and we may again extend it.

Consider the subpath $\delta$ of $\pi$ exceeding $|w|$. Then there are two distinct lifts of the concatenation $\delta a$ to $\Gamma$, one with image containing $v$ (the one via $\partial\Pi\to\Gamma$), and one with image containing $v'$ (this lift coincides on $a$ with the lift $\partial \Pi'\to\Gamma$). These lifts are essentially distinct since no automorphism of $\Gamma$ can take $v$ to $v'$. Therefore, $\delta a$ is an essential piece, whence the claim follows.
\end{proof}

\begin{cor}\label{cor:interior_degree3}
 Let $\Pi$ and $\Pi'$ be consecutive faces of $B_n$ such that $e(\Pi)=e(\Pi')=1$ and such that the exterior part of both is contained in $w^n$. Then not both have interior degree 3.
\end{cor}
\begin{proof}
 Assume that $|\Pi\cap w^n|>\frac{|\partial\Pi|}{2}$ and $|\Pi'\cap w^n|>\frac{|\partial\Pi'|}{2}$. (Otherwise the claim holds by the small cancellation assumption.) By Lemmas~\ref{lem:inequalities} and \ref{lem:sigma}, both are special, and $|\Pi|\geqslant 2|w|$ and $|\Pi'|\geqslant 2|w|$. Now Lemma~\ref{lem:consecutive} shows that, apart from the arc $a$ in the intersection of $\Pi$ and $\Pi'$, both $\Pi$ and $\Pi'$ must have three additional interior arcs, i.e. $i(\Pi)\geqslant 4$ and $i(\Pi')\geqslant 4$. 
\end{proof}

The following fact about diagrams follows immediately from \cite[Corollaries V.3.3 and V.3.4]{LS}.

\begin{lem}
 In a $(3,7)$-diagram, all faces are simply connected, and the intersection of any two faces is empty or connected.
\end{lem}

\begin{lem}\label{lem:connected} Let $D$ be a $(3,7)$-diagram. Let $\gamma$ be a subpath of $\partial D$ such that all faces incident at $\gamma$ whose exterior part is contained in $\gamma$ have interior degree at least 3, and no two consecutive such faces have interior degree less than 4. Then for any two faces $\Pi\neq\Pi'$ incident at $\gamma$, $\Pi\cap\Pi'$ is either empty or a connected path incident at $\gamma$. For any face $\Pi$, $\Pi\cap\gamma$ is empty or connected.
\end{lem}

\begin{proof}
Suppose in $D$ and $\gamma$ as above there exist $\Pi\neq\Pi'$ violating the first above statement. This means $D\setminus (\Pi\cup\Pi')$ has more than one connected component. Let $\Delta_0$ be a connected component of $D\setminus(\Pi\cup\Pi')$ with $\Delta_0\cap\gamma\neq\emptyset$, and consider the simple disk diagram $\Delta:=\overline\Delta_0$. Assume we chose $\Pi$ and $\Pi'$ such that $\Delta$ has minimal area among all possible choices. 

From $\Delta$ remove all faces incident at $\gamma$ to obtain a diagram $\Delta'$. Then, by minimality, $\Delta'$ is also a simple disk diagram. Consider a boundary face $f$ in $\Delta'$ with $e(f)=1$ (in $\Delta'$), and suppose $i(f)\leqslant 3$ (in $\Delta'$). Then, since the intersection of any two faces is connected and $f$ has degree at least 7 in $D$, $f$ intersects at least 4 faces from $\Pi,\Pi'$ and those faces of $\Delta$ incident at $\gamma$.
 
By minimality of $\Delta$, the intersection of any two faces of $\Delta\cup\Pi$ incident at $\gamma$ is empty or a connected path incident at $\gamma$, and the same holds for any two faces of $\Delta\cup\Pi'$. By the assumption, $f$ cannot intersect more than 3 faces of $\Delta$ incident at $\gamma$, and if it intersects 3 faces incident at $\gamma$, then $f$ intersects neither $\Pi$ nor $\Pi'$. 
 
Hence $f$ must intersect both $\Pi$ and $\Pi'$. Now there is at most one face such face $f$ in $\Delta'$ intersecting $\Pi,\Pi'$, and at least one face of $\Delta$ incident at $\gamma$. But by \cite[p. 241, Equation (8)]{Str}, any $(3,7)$-disk-diagram must contain at least two faces with exterior degree 1 and interior degree at most 3, a contradiction. 

The final statement follows with the same proof if the two faces $\Pi$ and $\Pi'$ are replaced by a single face $\Pi$.
\end{proof}

\begin{lem}\label{lem:consecutive2} Let $\Pi$ and $\Pi'$ be consecutive faces of $B_n$, where $|\Pi\cap w^n|\geqslant|w|$ and $\Pi'$ is not special. Then $\Pi'\cap w^n$ is a vertex or an essential piece. 
\end{lem}
\begin{proof}
 Note that $\Pi'\cap w^n$ is connected by Lemma~\ref{lem:connected}; assume it is not a vertex. Since $\Pi'$ is not special, we have $|\Pi'\cap w^n|<|w|$ by Lemma~\ref{lem:sigma}. Therefore, by assumption on $\Pi$, $\Pi'\cap w^n$ has a map to $\sigma$. Since $\Pi'$ is not special, we conclude that $\Pi'\cap w^n$ is an essential piece.
\end{proof}

We call a face $\Pi$ in $B_n$ \emph{very special} if the exterior part of $\Pi$ is entirely contained in $w^n$, $e(\Pi)=1$, and $i(\Pi)= 3$. Note that a very special face is, in particular, special by the proof of Corollary~\ref{cor:interior_degree3}. We make the following observations:

\begin{itemize}
 \item A very special face intersects no other special face in edges by Lemmas~\ref{lem:consecutive} and \ref{lem:connected}.
 \item Any non-special face incident at $w^n$ intersects at most two special faces in edges by Lemma~\ref{lem:connected}.
 \item A connected component of the intersection of any face not incident at $w^n$ with the union of all very special faces consists of at most one arc by Corollary~\ref{cor:interior_degree3} and Lemma~\ref{lem:connected}.
 \end{itemize}

\begin{lem}
 There is no very special face, and every disk component of $B_n$ is a either a single face or has shape $I_1$.
\end{lem}
\begin{proof}
From $B_n$, remove all very special faces and consider the remaining diagram $B_n'$. Denote $\partial B_n'=\beta_n\gamma_n^{-1}$ with immersed paths $\beta_n$ and $\gamma_n$, where $\gamma_n$ is isomorphic to $g_n$. We observe:
\begin{itemize}
 \item If $\Pi$ is a face of $B_n'$ that was special in $B_n$ and whose exterior part in $B_n'$ is an arc entirely contained in $\beta_n$, then $i(\Pi)\geqslant 4$. (This follows since special faces that are not very special are not affected at all by removing very special faces by the above observation.)
 \item If $\Pi$ is a face of $B_n'$ that was not special in $B_n$ and whose exterior part in $B_n'$ is an arc entirely contained in $\beta_n$, then $|\Pi\cap \beta_n|<\frac{|\partial\Pi|}{2}$, and, in particular, $i(\Pi)\geqslant 4$. (If $\Pi$ did not intersect any very special face of $B_n$ in an edge, this is immediate. Otherwise, Lemma~\ref{lem:consecutive2} and the above observations show that the arc in $\Pi\cap \beta_n$ consists of at most 3 essential pieces.)
\end{itemize}

Therefore, every disk component of $B_n'$ has shape $I_1$. Suppose there existed a very special face $\Pi$ of $B_n$. Then $\Pi$ intersected 3 faces $\Pi_1,\Pi_2,\Pi_3$ of $B_n'$ in arcs. Considering shape $I_1$, we see that at least one of them had a single boundary arc contained in $g_n$ and interior degree at most 3. This contradicts the fact that $g_n$ is geodesic. Thus we conclude $B_n=B_n'$.
\end{proof}

\begin{lem}
Every face $\Pi$ of $B_n$ satisfies $|\partial\Pi|<6C_0$ and $|\Pi\cap g_n|>0$.
\end{lem}
\begin{proof}
By construction of $C_0$, we have that $|\Pi\cap w^n|<C_0$. Since $|\Pi\cap g_n|\leqslant\frac{|\partial\Pi|}{2}$, and since $\Pi$ intersects at most two other faces, this implies the fist claim.

By Lemma \ref{lem:inequalities}, any face $\Pi$ intersects $w^n$ in less than $\frac{2|\partial\Pi|}{3}$. Therefore, $\Pi\cap g_n$ must contain an edge, whence the second claim follows.
\end{proof}

\begin{proof}[Proof of Theorem~\ref{thm:undistorted} in case 2b.] We conclude $|g_n|>\frac{|w^n|}{C_0}=n\frac{|w|}{C_0}$, and the image of $g_n$ in $\Cay(G(\Gamma),S)$ is within Hausdorff distance $6C_0+|w|$ of $\{1,w,\dots,w^n\}$. 
\end{proof}

For Theorem~\ref{thm:undistorted_product}, the very same proof applies. The only observation to be made is that if $\overline w$ is contained in the image of a $\Cay(G_i,S_i)$ for some $i$, then the cyclic subgroup generated by the conjugate of $g$ represented by $w$ is undistorted (respectively quasi-convex) in $\Cay(G,S)$ if and only if it is undistorted (respectively quasi-convex) in $\Cay(G_i,S_i)$ since any infinite $\Cay(G_i,S_i)$ is isometrically embedded and convex in $\Cay(G(\Gamma),S)$ by \cite[Remark 4.23]{GS}. \hfill $\square$